\documentclass[12pt,a4paper]{amsart}
\usepackage[all]{xy}
\usepackage{pdfsync}
 \usepackage[colorlinks=false]{hyperref}
\usepackage{amssymb,amsfonts} 
\usepackage{amsmath,systeme} 
\usepackage{graphicx} 
\usepackage[english]{babel} 
\usepackage{epsfig}
\usepackage{multirow}
\usepackage{color}
\usepackage{enumerate}
\usepackage[T1]{fontenc}
\usepackage{aurical, pbsi,calligra}

\usepackage[in]{fullpage}
\usepackage{tikz-network}

\newtheorem*{itheorem}{Theorem}
\theoremstyle{definition}
\newtheorem{definition}{Definition}[section]
\newtheorem{remark}[definition]{Remark}

\theoremstyle{plain}
\newtheorem{lemma}[definition]{Lemma}
\newtheorem{proposition}[definition]{Proposition}
\newtheorem{corollary}[definition]{Corollary}
\newtheorem{theorem}[definition]{Theorem}

\usepackage[all]{xy}
\usepackage{pdfsync}
\usepackage{ytableau}
\usepackage{rotating}

\usepackage[backend=bibtex,
    sorting=anyt,
    isbn=false,
    url=false,
    doi=false,
    maxbibnames=100
    ]{biblatex}
\usepackage{fourier}

\DeclareMathAlphabet{\pazocal}{OMS}{zplm}{m}{n}

\def\calO{\pazocal{O}}
\def\calP{\pazocal{P}}
\def\calQ{\pazocal{Q}}

\def\calS{\pazocal{S}}

\DeclareMathAlphabet{\mathbbold}{U}{bbold}{m}{n}

\def\k{\mathbbold{k}}

\DeclareMathOperator{\Com}{Com}
\DeclareMathOperator{\Perm}{Perm}
\DeclareMathOperator{\NAP}{NAP}

\begin{document}

\title{Distributive lattices of varieties of Novikov algebras}

\author{Vladimir Dotsenko}
\address{Institut de Recherche Math\'ematique Avanc\'ee, UMR 7501, Universit\'e de Strasbourg et CNRS, 7 rue Ren\'e-Descartes, 67000 Strasbourg CEDEX, France}
\email{vdotsenko@unistra.fr}

\author{Bekzat Zhakhayev}
\address{Institute of Mathematics and Mathematical Modeling, Pushkin St. 125, 050010 Almaty, Kazakhstan}
\email{bekzat.kopzhasar@gmail.com}

\dedicatory{To Leonid Arkadievich Bokut on the occasion of his 87th birthday}

\begin{abstract}
We prove that a variety of Novikov algebras has a distributive lattice of subvarieties if and only if the lattice of its subvarieties defined by identities of degree three is distributive, thus answering, in the case of Novikov algebras, a question of Bokut from about fifty years ago. As a byproduct, we classify all Koszul operads with one binary generator of which the Novikov operad is a quotient.
\end{abstract}

\maketitle

\section{Introduction}

Recall that a vector space over a field $\k$ equipped with a bilinear product $x,y\mapsto xy$ is called a \emph{(left) Novikov algebra} if the following identities hold for all $x,y,z\in N$:
\begin{gather*}
(x,y,z)=(x,z,y),\\
x(yz)=y(xz),
\end{gather*}
where $(x,y,z)=(xy)z-x(yz)$ is the associator of $x,y,z$. The term ``Novikov algebra'' was coined by Osborn \cite{MR1163779}. In fact, the identities of Novikov algebras seem to have first appeared in the study of Hamiltonian operators in the formal calculus of variations by Gelfand and Dorfman \cite{MR554407}, and then rediscovered by Balinskii and Novikov in the context of classification of linear Poisson brackets of hydrodynamical type~\cite{MR802121}. 

The main theorem of this article is a classification theorem for varieties of Novikov algebras whose lattice of subvarieties is distributive. The problem of classifying all varieties of algebras with a distributive lattice of subvarieties is recorded in the 1976 edition of Dniester Notebook by L. A. Bokut (see the easily accessible English translation of a later edition \cite[Question 1.179]{MR2203726}). Over a field of zero characteristic, this problem was solved for varieties of associative algebras by Anan'in and Kemer \cite{MR422112} and for varieties of alternative algebras and right-alternative algebras by Martirosyan \cite{MR654648,MR771888}. All those results can be stated in the following appealing way: in each of the cases of associative algebras, alternative algebras, and right-alternative algebras, a variety of algebras of tghat type has a distributive lattice of subvarieties if and only if the lattice of its subvarieties defined by identities of degree three is distributive. Note that this is not at all a general phenomenon: for instance, this is not true for the variety of Lie algebras, where the first obstruction to distributivity appears among identities of degree six. However, our main result asserts that this is the case for varieties of Novikov algebras. Specifically, we prove the following theorem.  
\begin{itheorem}[{Th.~\ref{th:maindistr}}]
The lattice of subvarieties of a variety of Novikov algebras is distributive if and only if all algebras of that variety satisfy the identities
 \[
\alpha a^2a+\beta aa^2,\quad
\gamma((a,a,b)-(b,a,a))+\delta(a(ab)-ba^2)
 \]
for some $((\alpha:\beta),(\gamma:\delta))\in\mathbb{P}^1\times\mathbb{P}^1$.
\end{itheorem}

The abovementioned result of Anan'in and Kemer was refined by Drenski and Vladimirova \cite{MR866654} who studied in great detail varieties of associative algebras defined by identities of degree and their lattices of subvarieties. These latter results were recently used by Bremner and the first author of this paper to classify Koszul quotients of the associative operad in \cite{MR4576938}. Similarly, we were able to use the main result of the present paper to classify Koszul quotients of the Novikov operad. Since the Novikov operad is isomorphic to its Koszul dual, this also gives a classification of Koszul operads with one binary generator of which the Novikov operad is a quotient. Recall that Dzhumadildaev \cite{MR2834140} proved that the Novikov operad is not Koszul, so this result describes all ways in which a Novikov algebra can be regarded as an algebra over an Koszul operad. 
(This may be compared with a similar problem of Loday \cite{MR3013086} asking to determine Koszul operads that act on the algebra of octonions, a question that motivated the paper~\cite{MR4576938}.) Specifically, we prove the following theorem.

\begin{itheorem}[{Th.~\ref{th:Koszul}}]
The following Koszul operads with one binary generator admit the (right) Novikov operad as a quotient:
\begin{itemize}
\item the operad of (left) nonassociative permutative algebras $\NAP$ defined by the identity $a_1(a_2a_3)-a_2(a_1a_3)=0$,
\item the (right) pre-Lie operad defined by the identity $(a_1,a_2,a_3)=(a_1,a_3,a_2)$,
\item each operad in the parametric family depending on the parameter $(\gamma:\delta)\in\mathbb{P}^1$ defined by the identity 
\begin{multline*}
\gamma((a_1,a_2,a_3)+(a_3,a_2,a_1)-(a_2,a_1,a_3)-(a_2,a_3,a_1))+\\
\delta((a_1a_2)a_3+(a_3a_2)a_1-(a_1a_3)a_2-(a_3a_1)a_2),
\end{multline*}
\item each operad in the parametric family depending on the parameter $(\alpha:\beta)\in\mathbb{P}^1$ defined by the identity 
\begin{multline*}
\alpha((a_1a_2)a_3+(a_2a_3)a_1+(a_3a_1)a_2+(a_1a_3)a_2+(a_2a_1)a_3+(a_3a_2)a_1)+\\
\beta(a_1(a_2a_3)+a_2(a_3a_1)+a_3(a_1a_2)+a_1(a_3a_2)+a_2(a_1a_3)+a_3(a_2a_1)).
\end{multline*}
\item the magmatic operad of absolutely free nonassociative algebras.
\end{itemize}
\end{itheorem}

This paper is organized as follows. In Section \ref{sec:recollections} we give some necessary recollections of definitions and results we use. In Section \ref{sec:distrthm}, we undertake a systematic study of quotients of the Novikov operad: we obtain information on $S_n$-module structures on their components, which allows us to prove our main result, the classification of varieties of Novikov algebras whose lattice of subvarieties is distributive. In Section \ref{sec:koszulthm}, we use the results obtained in the previous section to classify all Koszul operads admitting the Novikov operad as a quotient. Finally, in Appendix we prove several technical computational results that we use in the paper.

\subsection*{Acknowledgements} We would like to dedicate this article to Leonid Arkadievich Bokut on the occasion of his 87th birthday. His  passion for algebra in all its richness has been a constant inspiration for several generations of mathematicians, and we are excited to have contributed to the investigation of the very interesting question that he raised, that of describing varieties of algebras whose lattice of subvarieties is distributive. We wish him a wonderful birthday and many happy returns of the day.

The first author was supported by Institut Universitaire de France. The second author was supported by the Kazakhstan Presidential Bolashak Scholarship Program that supported his stay at Institut de Recherche Mathématique Avancée at the University of Strasbourg. The final draft of this paper was completed while the first author was visiting the Banach center in Warsaw during the Simons semester ``Knots, homologies, and physics''. They wish to express their gratitude to those institutions for hospitality and excellent working conditions. 

\section{Recollections}\label{sec:recollections}

Throughout this paper, all vector spaces are defined over an arbitrary field $\k$ of characteristic zero. By an \emph{algebra} we understand a vector space $V$ equipped with several multilinear structure operations $f_i\colon V^{\otimes n_i}\to V$; here $n_i$ is the \emph{arity} of the operation $f_i$. If one fixes a set of structure operations $S$, all algebras with such operations form a category, and one has the forgetful functor from that category to the category of vector spaces. That functor admits a left adjoint, applying that functor to a vector space $U$ is the \emph{absolutely free algebra} generated by $U$, and denoted $F_S\langle U\rangle$. That algebra has a basis of \emph{monomials} that are iterations of the structure operations applied to elements of a basis of $U$; a linear combination of monomials in the absolutely free algebra will be referred to as a \emph{polynomial}. A \emph{polynomial identity} in $m$ variables in an algebra $V$ is a polynomial in the absolutely free algebra $F_S\langle \k^m\rangle$ that vanishes under any algebra morphism $F_S\langle \k^m\rangle\to V$ corresponding, via the adjunction, to a linear map $\k^m\to V$ (or, in plain words, vanishes under any substitutions of elements of $V$ instead of its arguments). A \emph{variety of algebras} is a subcategory all algebras with the given set of structure operations where a certain set of polynomial identities is satisfied. We refer the reader to~\cite{MR668355} for general information on polynomial identities.

Recall that a \emph{lattice} is a poset in which every two elements $a,b$ the set of elements that are less than both of them admits the unique maximal element $a\wedge b$ and the set of elements that are greater than both of them admits the unique minimal element $a\vee b$. For a variety of algebras $\mathfrak{M}$, all its subvarieties form a lattice with respect to inclusion. Here $\mathfrak{M}_1\wedge\mathfrak{M}_2$ consists of all algebras where all identities defining each of the varieties $\mathfrak{M}_1$ and $\mathfrak{M}_2$ are satisfied, and $\mathfrak{M}_1\vee\mathfrak{M}_2$ consists of all algebras where all identities that hold in both varieties $\mathfrak{M}_1$ and $\mathfrak{M}_2$ are satisfied. 

Recall that over a field of characteristic zero every representation of the symmetric group $S_n$ is completely reducible, and that its irreducible representations $V_\pi$ are indexed by partitions $\pi$ of $n$ (that is, $\pi=(m_1,\ldots,m_k)$ with $m_1\geq \ldots\geq m_k$ and $m_1\cdots+m_k=n$). The reader is invited to consult \cite{MR3583300} for a detailed introduction to representation theory of symmetric groups in the context of polynomial identities.

A polynomial identity in $F_S\langle \k^m\rangle$ is said to be \emph{multihomogeneous} of degree $(d_1,\ldots,d_m)\in\mathbb{N}^m$ if it is a linear combination of monomials that contain the $i$-th generator $d_i$ times for all $i=1,\ldots,m$. In particular, a polynomial identity is said to be \emph{multilinear} if it is multihomogeneous of degree $(1,1,\ldots,1)$. The following result is well known (its first part is true over any infinite field, not necessarily of zero characteristic). 

\begin{proposition}\label{prop:multlin}
Every identity is equivalent to a system of multihomogeneous ones in the same variables. Moreover, every identity is equivalent to a system of multilinear ones (in a larger number of variables).
\end{proposition}

The passage from multihomogeneous to multilinear is done by using derivations $\Delta_{a_i\mapsto b}$ of the free algebra (that send the generator $a_i$ to a new generator $b$ and all other generators to zero). Such a derivation sends an identity of degree $d_i>1$ in $a_i$ to an equivalent identity of degree $d_i-1$ in $a_i$, and an inductive argument completes the proof. These and similar derivations will be extensively used in this paper. 

There is also a useful argument (which we shall also use extensively in this paper) going from multilinear identities to multihomogeneous ones which can be traced to the ``Aronhold polarization process'' \cite{MR1579064}. 

\begin{proposition}\label{prop:sym-id}
Let $f$ be a multihomogeneous polynomial identity of degree $(d_1,\ldots,d_n)$ in variables $a_1,\ldots,a_n$, and suppose that for some $k\le n$ we have $d_1=\ldots=d_k=1$ and, additionally, $f$ is symmetric in $a_1,\ldots, a_k$. Then the multihomogeneous identity obtained from $f$by setting $a_1=\cdots=a_k$ is equivalent to $f$.
\end{proposition}

Proposition \ref{prop:multlin} leads to a way of thinking of varieties of algebras in terms of \emph{operads}. An operad is a collection of representations of symmetric groups equipped with operations that mimic substitutions of multilinear maps and satisfy the same properties that such substitutions satisfy. Operads were first defined by J.~P.~May in 1971 in his work on iterated loop spaces \cite{MR0420610}, the same notion seems to have been first introduced under a much more technical name of a ``clone of multilinear operations'' in Artamonov's 1969 paper \cite{MR0237408}. Operads are in one-to-one correspondence with varieties of algebras; however, the language of operads allows to use some methods that are not available on the level of algebras. Perhaps one of the most powerful method of that sort is the theory of operadic Gr\"obner bases \cite{MR2667136}, which goes via the notion of a \emph{shuffle operad} that cannot be defined intrinsically on the level of varieties of algebras. 
For example, if we take the associative operad, as a symmetric operad it is generated by a single operation $a_1,a_2\mapsto a_1a_2$ subject to the single relation $(a_1a_2)a_3=a_1(a_2a_3)$. In the universe of shuffle operads, one has to forget the symmetric groups actions, and write linear bases both for generators and for relations in terms of shuffle tree monomials \cite[Sec.~5.3]{MR3642294}, \cite[Sec.~8.2]{MR2954392}, which gives two generators and six relations in the case of the associative operad. We refer the reader to \cite{MR2954392} for general information on operads, to \cite{MR3642294} for a hands-on introduction to operadic Gr\"obner bases, and to \cite[Sec.~2]{DU22} for a discussion of translation between the language of varieties of algebras and the language of operads.

An important class of varieties of algebras consists of varieties whose subvarieties form a distributive lattice. Recall that a lattice is said to be \emph{distributive} if it satisfies $(x\vee y)\wedge z=(x\wedge z)\vee(y\wedge z)$. The following criterion of distributivity in terms of representations of symmetric groups proved in \cite{MR495500} (and rephrased here using operads) will be extensively used in this paper.

\begin{proposition}\label{prop:distr}
Let $\mathfrak{M}$ be a variety of algebras, and $\calO_\mathfrak{M}$ be an operad describing that variety. 
The lattice of subvarieties of $\mathfrak{M}$ is distributive if and only if for each $n$ the $S_n$-module $\calO_\mathfrak{M}(n)$ contains each irreducible representation with multiplicity at most one.
\end{proposition}

The language of operads is also useful in questions of homological or homotopical nature, where the theory of Koszul duality for operads \cite{MR1301191} has particular prominence. This theory is only applicable if an operad is Koszul, and determining that is often a very nontrivial question. To prove that an operad is Koszul, the easiest and most general known approach is to use operadic Gr\"obner bases: a shuffle operad that has a quadratic Gr\"obner basis is known to be Koszul \cite[Sec.~6.4]{MR3642294}, though the converse is false. Moreover, the same argument can be used to show that an operad presented by a convergent quadratic rewriting system \cite[Sec.~2.6]{MR3642294} is Koszul. Finding a suitable rewriting system is sometimes a matter of luck, as it heavily depends on the choice of a presentation by generators and relations. (It is worth noting that, for operads generated by one binary operation, there is a useful ``polarization trick'' \cite{MR2225770} that introduces a  presentation by generators and relations which is sometimes preferable: it amounts to considering the generators $a_1\cdot a_2=a_1a_2+a_2a_1$ and $[a_1,a_2]=a_1a_2-a_2a_1$.) 

To prove that an operad is not Koszul, one often ends up using Poincar\'e series, that is exponential generating functions of Euler characteristics of components of our operad. For an operad $\calP$ concentrated in homological degree zero, the Poincar\'e series coincides with the Hilbert series
 \[
f_{\calP}(t)=\sum_{n\ge 1}\frac{\dim\calO(n)}{n!}t^n.
 \] 
By a direct inspection, one sees that the Poincar\'e series of the Koszul complex of a quadratic operad generated by binary operations of homological degree zero is equal to $-f_{\calP^!}(-f_{\calP}(t))$. Since the Euler characteristics of a chain complex and its homology are equal, this implies that for a Koszul operad $\calP$, one has
 \[
-f_{\calP^!}(-f_{\calP}(t))=t,
 \]
so the series $f_{\calP}(t)$ and $-f_{\calP^!}(-t)$ are compositional inverses of one another. This leads to a useful positivity test of Ginzburg and Kapranov \cite{MR1301191}.   

\begin{proposition}[Positivity test]\label{prop:positivity}
Let $\calP$ be a quadratic operad generated by binary operations of homological degree zero. Denote by $a_n$ the coefficient of $t^n$ in the compositional inverse of the Poincar\'e series of that operad. If the operad $\calP$ is Koszul, then $(-1)^{n-1}a_n\ge 0$ for all $n\ge 1$.
\end{proposition}

There is also the following useful sufficient condition of Koszulness in terms of Poincar\'e series; the reader is invited to consult \cite[Prop.~2.4]{MR4576938} for a proof.

\begin{proposition}\label{prop:dim2}
Let $\calP$ be a quadratic operad generated by binary operations of homological degree zero. Suppose that $\calP(n)=0$ for $n\ge 4$, and that
 \[
-f_{\calP^!}(-f_{\calP}(t))=t.
 \]
Then the operad $\calP$ is Koszul.
\end{proposition}

\section{Consequences of degree three identities and distributivity}\label{sec:distrthm}

In this section, we shall use Proposition \ref{prop:distr} to classify all varieties of Novikov algebras whose lattice of subvarieties is distributive. Since the arity three component of the Novikov operad is a direct sum of irreducible modules $V_3$ and $V_{2,1}$, each with multiplicity two (see, e.g. \cite{MR3241181}), in order to obtain a distributive lattice one must quotient out a copy of each of them. We study the corresponding quotients individually, and then use the corresponding results to show that there are no further obstructions for distributivity. For that, we shall use represent Novikov algebras as subalgebras of commutative associative differential algebras (this goes back to work of I.~M.~Gelfand and I.~Ja.~Dorfman~\cite{MR554407} who attribute this construction to S.~I.~Gelfand). It is well known that if $A$ is a commutative associative algebra with a derivation $a\mapsto a'$, the product $a'b$ makes $A$ into a Novikov algebra. Moreover, it is proved by Dzhumadildaev and L\"ofwall \cite[Th.~7.8]{MR1918188} that the free Novikov algebra can be realized as a subalgebra of the free commutative associative differential algebra (spanned by all \emph{differential monomials} $a_{k_1}^{(i_1)}a_{k_2}^{(i_2)}\cdots a_{k_n}^{(i_n)}$, see \cite{Kolchin,Ritt}) spanned by the differential monomials for which $i_1+\cdots+i_n=n-1$. We shall refer to these as \emph{Novikov differential monomials}. This result means that, when working with free Novikov algebras, we may perform various calculations in free commutative associative differential algebras, and their basis and structure constants are much more intuitive than those of free Novikov algebras. (In fact, it was proved by Bokut, Chen, and Zhang \cite{BCZ18} that every Novikov algebra embeds into an appropriate differential enveloping algebra, so one can faithfully represent any Novikov algebra this way.) This was already used in \cite{MR4549103} to show that every collection of identites in Novikov algebras follows from finitely many of them. It is crucial to important to preserve the defining property of Novikov differential monomials: when deriving new identities from an identity $f=0$, we can replace it by $f'a$, $fa'$, or by a result of substitution $ab'$ instead of one of the variables.

\subsection{Quotienting out a copy of the trivial module}\label{sec:triv}

Let $(\alpha:\beta)\in\mathbb{P}^1$. We consider the operad $\calP_{\alpha,\beta}$ that is the quotient of the Novikov operad by the ideal generated by the identity 
 \[
\alpha a^2a+\beta aa^2=0.
 \]
In the differential realization, this is the identity
\begin{equation}\label{eq:3}
\alpha a''a^2+(\alpha+\beta) (a')^2a =0,
\end{equation}
We shall now determine how the $S_n$-module structure of $\calP_{\alpha,\beta}(n)$ depends on $(\alpha:\beta)$. Clearly, for all $(\alpha:\beta)\in\mathbb{P}^1$, we have 
 \[
\calP_{\alpha,\beta}(1)\cong V_{1}, \quad \calP_{\alpha,\beta}(2)\cong V_{2}\oplus V_{1,1}, \quad \calP_{\alpha,\beta}(3)\cong V_{3}\oplus V_{2,1}^2.
 \]

Namely, we prove the following theorem.

\begin{theorem}\label{th:rel1}
Let $(\alpha:\beta)\in\mathbb{P}^1$, and let $n\ge 4$. The $S_n$-module structure of $\calP_{\alpha,\beta}(n)$ is described as follows:
\begin{itemize}
\item for $(\alpha:\beta)=(0:1)$, we have $\calP_{\alpha,\beta}(n)\cong V_n\oplus V_{n-1,1}$, and these modules are generated by linearizations of $a^{(n-1)}a^{n-1}$ and $a^{(n-1)}ba^{n-2}-b^{(n-1)}a^{n-1}$, respectively. 
\item for $(\alpha\colon\beta)=(1:1)$, we have
 \[
\calP_{\alpha,\beta}(n)\cong
\begin{cases}
V_{3,1}\oplus V_{2,2}\oplus V_{2,1,1}, \,\,\,\,\quad n=4,\\
\qquad\quad V_{2,2,1}, \qquad\qquad n=5,\\
\qquad\qquad 0, \qquad\qquad\quad \! n\ge 6,
\end{cases}
 \]
and these modules are generated by linearizations of
 \[
a''a'ab-a''b'a^2, \ a''a'b^2-a''b'ab-b''a'ab+b''b'a^2,
 \] 
 \[
a''b'ca-a''c'ba-b''a'ca+b''c'a^2+c''a'ba-c''b'a^2, 
 \]
and
 \[
(a''b'c-a''c'b-b''a'c+b''c'a+c''a'b-c''b'a)(a'b-b'a)
 \]
respectively.
\item for $(\alpha\colon\beta)=(1:-1)$, we have $\calP_{\alpha,\beta}(n)\cong V_n\oplus V_{n-1,1}$, and these modules are generated by   linearizations of $(a')^{n-1}a$ and $(a')^{n-1}b-b'(a')^{n-2}a$, respectively.
\item otherwise, we have $\calP_{\alpha,\beta}(n)=0$.
\end{itemize}
\end{theorem}

\begin{proof}
Multiplying \eqref{eq:3} by $a'$, we get
\begin{equation}\label{eq:A05}
\alpha a''a'a^2+(\alpha+\beta) (a')^2a=0.
\end{equation}
Taking the derivative of \eqref{eq:3} and multiplying by $a$, we get
\begin{equation}\label{eq:A06}
\alpha (a'''a^3+2a''a'a^2)+(\alpha+\beta) (2a''a'a^2+(a')^3a)=0.
\end{equation}
Applying the derivation $\Delta_{a\mapsto a'a}$ to \eqref{eq:3}, we get 
\begin{equation}\label{eq:A07}
\alpha (a'''a^3+5a''a'a^2)+(\alpha+\beta )(2a''a'a^2+3(a')^3a)=0.
\end{equation}
Overall, we obtain 
\begin{equation}\label{eq:mat1}
\begin{pmatrix}
0&\alpha& \alpha+\beta\\
\alpha&4\alpha+2\beta&\alpha+\beta\\
\alpha&7\alpha+2\beta&3\alpha+3\beta
\end{pmatrix}
\begin{pmatrix}
a'''a^3\\
a''a'a^2\\
(a')^3a
\end{pmatrix}=0.
\end{equation}
Note that
 \[
\det 
\begin{pmatrix}
0&\alpha& \alpha+\beta\\
\alpha&4\alpha+2\beta&\alpha+\beta\\
\alpha&7\alpha+2\beta&3\alpha+3\beta
\end{pmatrix}=
\alpha^2(\alpha+\beta),
 \]
so for $(\alpha:\beta)$ different from $(0:1)$ and $(1:-1)$, we have $a'''a^3=a''a'a^2=(a')^3a=0$. Let us assume for the time being that
$(\alpha:\beta)$ is different from $(0:1)$ and $(1:-1)$, deferring these cases to Propositions \ref{prop:case01} and \ref{prop:case1-1} respectively. Partial multilinearizations of the identities we obtained are the identities
\begin{gather}
b''a'a^2+a''b'a^2+2a''a'ab=0,\label{eq:A16}\\
3(a')^2b'a+(a')^3b=0,\label{eq:A17}\\
b'''a^3+3a'''a^2b=0. \label{eq:A18}
\end{gather}
Multiplying \eqref{eq:3} by $b'$, we get
\begin{equation}\label{eq:A19}
\alpha a''b'a^2+(\alpha+\beta) (a')^2b'a=0.
\end{equation}
Taking the derivative of \eqref{eq:3} and multiplying by $b$, we get
\begin{equation}\label{eq:A20}
\alpha (a'''a^2b+2a''a'ab)+(\alpha+\beta)(2 a''a'ab+(a')^3b)=0.
\end{equation}
Applying the derivation $\Delta_{a\mapsto a'b}$ to \eqref{eq:3}, we get 
\begin{equation}\label{eq:A21}
\alpha (a'''a^2b+2a''b'a^2+b''a'a^2+2a''a'ab)+(\alpha+\beta) (2a''a'ab+2(a')^2b'a+(a')^3b)=0.
\end{equation}
Applying the derivation $\Delta_{a\mapsto b'a}$ to \eqref{eq:3}, we get 
\begin{equation}\label{eq:A22}
\alpha (b'''a^3+2b''a'a^2+3a''b'a^2)+(\alpha+\beta)(2b''a'a^2+3(a')^2b'a)=0.
\end{equation}

Overall, we obtain
\begin{equation}\label{eq:mat2}
\begin{pmatrix}
0&0&2&1&1&0&0\\
0&0&0&0&0&1&3\\
3&1&0&0&0&0&0\\
0&0&0&\alpha&0&0&\alpha+\beta\\
\alpha&0&4\alpha+2\beta&0&0&\alpha+\beta&0\\
\alpha&0&4\alpha+2\beta&2\alpha&\alpha&\alpha+\beta&2\alpha+2\beta\\
0&\alpha&0&3\alpha&4\alpha+2\beta&0&3\alpha+3\beta
\end{pmatrix}
\begin{pmatrix}
a'''a^2b\\
b'''a^3\\
a''a'ab\\
a''b'a^2\\
b''a'a^2\\
(a')^3b\\
b'(a')^2a
\end{pmatrix}=0.
\end{equation}
Note that 
 \[
\det
\begin{pmatrix}
0&0&2&1&1&0&0\\
0&0&0&0&0&1&3\\
3&1&0&0&0&0&0\\
0&0&0&\alpha&0&0&\alpha+\beta\\
\alpha&0&4\alpha+2\beta&0&0&\alpha+\beta&0\\
\alpha&0&4\alpha+2\beta&2\alpha&\alpha&\alpha+\beta&2\alpha+2\beta\\
0&\alpha&0&3\alpha&4\alpha+2\beta&0&3\alpha+3\beta
\end{pmatrix}\\=
6\alpha^2(\beta-\alpha)(\beta+\alpha),
 \]
so if $(\alpha:\beta)$ is additionally different from $(1:1)$, we obtain 
 \[
a'''a^2b=b'''a^3=a''a'ab=a''b'a^2=b''a'a^2=(a')^3b=b'(a')^2a=0.
 \]
Let us additionally assume that $(\alpha:\beta)\ne(1:1)$. From $a'''a^3=a'''a^2b=b'''a^3=0$ one immediately deduces the identity
$a'''bcd=0$, since it is symmetric in $b,c,d$, and hence follows from its versions where at most two letters are different. For the same reason, from $(a')^3a=(a')^3b=b'(a')^2a=0$, one immediately deduces $a'b'c'd=0$.

Applying the derivation $\Delta_{a\mapsto b}$ to $b''a'a^2=0$, one obtains 
\begin{equation}\label{eq:A39}
b''b'a^2+2b''a'ab=0.
\end{equation}
Applying the derivation $\Delta_{a\mapsto b}$ to $a''a'ab=0$, one obtains 
\begin{equation}\label{eq:A40}
b''a'ab+a''b'ab+a''a'b^2=0.
\end{equation}
Applying the derivation $\Delta_{a\mapsto b}$ to $a''b'a^2=0$, one obtains 
\begin{equation}\label{eq:A41}
b''b'a^2+2a''b'ab=0.
\end{equation}
Applying the derivation $\Delta_{a\mapsto b'b}$ to \eqref{eq:3} and using $a'b'c'd=0$, one obtains 
\begin{equation}\label{eq:A42}
\alpha (3b''b'a^2+2a''b'ab)+(\alpha+\beta) (2b''a'ab)=0.
\end{equation}
Overall, we obtain 
\begin{equation}\label{eq:mat3}
\begin{pmatrix}
1&1&1&0\\
0&2&0&1\\
0&0&2&1\\
0&2\alpha&2\alpha+2\beta&3\alpha
\end{pmatrix}
\begin{pmatrix}
a''a'b^2\\
a''b'ab\\
b''a'ab\\
b''b'a^2
\end{pmatrix}=0.
\end{equation}
Since 
 \[
\det
\begin{pmatrix}
1&1&1&0\\
0&2&0&1\\
0&0&2&1\\
0&2\alpha&2\alpha+2\beta&3\alpha
\end{pmatrix}
=4(\alpha-\beta)\ne 0,
 \]
we have $a''a'b^2=a''b'ab=b''a'ab=b''b'a^2=0$.

Applying the derivation $\Delta_{b\mapsto c}$ to $b''a'ab=0$, one obtains 
\begin{equation}\label{eq:A49}
c''a'ab+b''a'ac=0.
\end{equation}
Applying the derivation $\Delta_{b\mapsto c}$ to $b''b'a^2=0$, one obtains 
\begin{equation}\label{eq:A50}
c''b'a^2+b''c'a^2=0.
\end{equation}
Applying the derivation $\Delta_{b\mapsto c}$ to $a''b'ab=0$, one obtains 
\begin{equation}\label{eq:A51}
a''c'ab+a''b'ac=0.
\end{equation}
Applying the derivation $\Delta_{a\mapsto b'c}$ to \eqref{eq:3}, and using $a'''bcd=a'b'c'd=0$ and \eqref{eq:A50}, one obtains 
\begin{equation}\label{eq:A52}
\alpha (b''c'a^2+2a''b'ac)+(\alpha+\beta) (2b''a'ac)=0.
\end{equation}
Applying the derivation $\Delta_{a\mapsto b}$ to \eqref{eq:3}, multiplying by $c'$, and using $a'b'c'd=0$, one obtains, recalling that $\alpha\ne 0$, 
\begin{equation}\label{eq:A54}
b''c'a^2+2a''c'ab=0.
\end{equation}
Applying the derivation $\Delta_{a\mapsto c}$ to $b''a'a^2=0$, one obtains 
\begin{equation}\label{eq:A56}
b''c'a^2+2b''a'ac=0.
\end{equation}
Overall, we have 
\begin{equation}\label{eq:mat4}
\begin{pmatrix}
0&0&1&1&0&0\\
0&0&0&0&1&1\\
1&1&0&0&0&0\\
2\alpha&0&2\alpha+2\beta&0&\alpha&0\\
0&2&0&0&1&0\\
0&0&2&0&1&0
\end{pmatrix}
\begin{pmatrix}
a''b'ac\\
a''c'ab\\
b''a'ac\\
c''a'ab\\
b''c'a^2\\
c''b'a^2
\end{pmatrix}=0
\end{equation}
Since 
  \[
\det
\begin{pmatrix}
0&0&1&1&0&0\\
0&0&0&0&1&1\\
1&1&0&0&0&0\\
2\alpha&0&2\alpha+2\beta&0&\alpha&0\\
0&2&0&0&1&0\\
0&0&2&0&1&0
\end{pmatrix}
=4(\beta-\alpha)\ne 0,
 \]
we have 
 \[
a''b'ac=
a''c'ab=
b''a'ac=
c''a'ab=
b''c'a^2=
c''b'a^2=0.
 \] 
Finally, from 
 \[
a''a'a^2=a''a'ab=a''ba^2=b''a'a^2=a''b'ac=b''a'ac=b''c'a^2=0
 \] 
one immediately deduces $a''b'cd=0$, since it is symmetric in $c,d$, and hence follows from its versions where at most three letters are different. This shows that in the ``generic'' case $\alpha(\alpha-\beta)(\alpha +\beta)\neq 0$ we have $\calP_{\alpha,\beta}(n)=0$ for $n>4$.

We shall now return to the case $(\alpha:\beta)=(1:1)$ that was temporarily put aside. Recall that in this case we have $a'''a^3=a''a'a^2=(a')^3a=0$, and Equation \eqref{eq:mat2} becomes
 \[
\begin{pmatrix}
0&0&2&1&1&0&0\\
0&0&0&0&0&1&3\\
3&1&0&0&0&0&0\\
0&0&0&1&0&0&2\\
1&0&6&0&0&2&0\\
1&0&6&2&1&2&4\\
0&1&0&3&6&0&6
\end{pmatrix}
\begin{pmatrix}
a'''a^2b\\
b'''a^3\\
a''a'ab\\
a''b'a^2\\
b''a'a^2\\
(a')^3b\\
b'(a')^2a
\end{pmatrix}=0.
 \]
Elementary row operations easily give us monomial relations $a'''a^2b=b'''a^3=b''a'a^2=0$, and three slightly more complicated relations, namely
\begin{gather}
a''b'a^2+2a''a'ab=0,\label{eq:A114}\\
a''b'a^2+2b'(a')^2a=0,\label{eq:A115}\\
(a')^3b+3b'(a')^2a. \label{eq:A116}
\end{gather}

To show that $\calP_{1,1}(n)=0$ for $n\ge 6$, it is enough to show that for $n=6$, which in turn would follow from the fact that $a_1^{(k_1)}a_2^{(k_2)}\cdots a_6^{(k_6)}=0$ whenever $k_1+\cdots +k_n=5$, $k_1\geq \ldots \geq k_n\geq 0$. 

First, we note that $a'''a^3=a'''a^2b=b'''a^3=0$ imply
$a'''bcd=0$, since it is symmetric in $b,c,d$, and hence follows from its versions where at most two letters are different. This immediately implies $a^{(3)}b'c'def=0$ (multiplying by derivatives), 
 \[
a^{(4)}bcde=(a'''bcd)'e-(a'''cde)b'-(a'''bde)c'-(a'''bce)d'=0, 
 \]
which in turn implies $a^{(4)}b'cdef=0$ and  
 \[
a^{(5)}bcdef=(a^{(4)}bcde)'f-(a^{(4)}cdef)b'-(a^{(4)}bdef)c'-(a^{(4)}bcef)d'-(a^{(4)}bcdf)e'=0. 
 \]
We also have 
 \[
0=(a'''b'cde)'f=a^{(4)}b'cdef+a'''b''cdef+a'''b'c'def+a'''b'cd'ef+a'''b'cde'f,
 \]
implying $a'''b''cdef=0$. 

If we multiply \eqref{eq:A116} by $a'$ and using $(a')^3a=0$, we obtain $(a')^4b=0$. On the other hand, if we multiply $(a')^3a=0$ by $b'$, we obtain $(a')^3b'a=0$. From $(a')^3a=(a')^4b=(a')^3b'a=0$ one immediately deduces 
$a'b'c'd'e=0$, since it is symmetric in $a,b,c,d$, and hence follows from its versions where at most two letters are different. This implies $a'b'c'd'e'f=0$. 
Substuting $a=a'f$ into $a'b'c'd'e=0$, we obtain
 \[
0=(a'f)'b'c'd'e=a''b'c'd'ef+a'b'c'd'ef',
 \]
implying $a''b'c'd'ef=0$.

Taking the derivative of $a''a'a^2=0$ and multiplying by $a$,  we obtain, using $a'''bcd=0$ and $a''a'a^2=0$,
$(a'')^2a^3=0$. Multiplying $b''a'a^2=0$ by $a'$, we obtain $b''(a')^2a^2=0$, which we can in turn use to simplify the result of applying the derivation $\Delta_{a\mapsto a'a}$ to $b''a'a^2=0$, obtainining $a''b''a^3=0$. Furthermore, this latter relation can be used to simplify the result of applying the derivation $\Delta_{a\mapsto b}$ to $(a'')^2a^3=0$, obtaining $(a'')^2a^2b=0$. Applying the derivation $\Delta_{a\mapsto b'a}$ to $b''a'a^2=0$ and simplifying, we obtain $(b'')^2a^3=0$, which we can use to simplify the result of applying the derivation $\Delta_{a\mapsto b}$ to $a''b''a^3=0$, obtaining $a''b''a^2b=0$. Furthermore, we can use that latter equation to simplify the result of applying the derivation $\Delta_{a\mapsto b}$ to $(a'')^2a^2b=0$, obtaining $(a'')^2ab^2=0$.
From 
 \[
(a'')^2a^3=a''b''a^3=(a'')^2a^2b=(b'')^2a^3=a''b''a^2b=(a'')^2ab^2=0,
 \]
one immediately deduces $a''b''cde=0$, since it is symmetric in $a,b$ and in $c,d,e$, and hence follows from its versions where at most two letters are different. This immediately implies $a''b''c'def=0$. All these identities imply that $\calP_{1,1}(n)=0$ for $n\ge 6$, as required. 

To prove the claims about $\calP_{1,1}(4)$ and $\calP_{1,1}(5)$, some further calculations are needed. In arity $4$, differential Novikov monomials correspond to partitions of $3$, that is $(3), (2,1), (1,1,1)$. We already established that $a'''bcd=0$, so it is enough to consider the submodules generated by $S_4$-orbits of $a''b'cd$ and $a'b'c'd$. In the Novikov operad these monomials generate $S_4$-submodules isomorphic to $V_4\oplus V_{3,1}^2\oplus V_{2,2}\oplus V_{2,1,1}$ and $V_4\oplus V_{3,1}$, respectively. Relations $a''a'a^2=(a')^3a=0$ imply that there are no copies of $V_4$ in $\calP_{\alpha,\beta}$. Relation $b''a'a^2=0$ quotients out one copy of $V_{3,1}$. What remains is precisely $V_{3,1}\oplus V_{2,2}\oplus V_{2,1,1}$, and to conclude that no further elements vanish in the quotient, one may compute the dimension of $\calP_{1,1}(4)$ using the \texttt{albert} program \cite{albert} or the operad Gr\"obner basis calculator \cite{OpGb}. A similar argument applies in arity $5$: differential Novikov monomials correspond to partitions of $4$, that is $(4), (3,1), (2,2), (2,1,1), (1,1,1,1)$, and we already established that $a^{(4)}bcde=a^{(3)}b'cde=a''b''cde=a'b'c'd'e=0$, so we should focus on the $S_5$-orbit of the monomial $a''b'c'de$. In the Novikov operad this monomial generates an $S_5$-submodules isomorphic to $V_5\oplus V_{4,1}^2\oplus V_{3,2}^2\oplus V_{3,1,1}\oplus V_{2,2,1}$. Our previous computations easily show that the versions of this monomial where at most two letters are different vanish in the quotient, implying quotienting out $V_5\oplus V_{4,1}^2\oplus V_{3,2}^2$. Computing the dimension of $\calP_{1,1}(5)$ using the abovementioned software, we find that it is equal to five, implying that it is the module $V_{2,2,1}$ survives in the quotient (since the dimension of $V_{3,1,1}$ is six).  

\smallskip

Let us return to the cases $(\alpha:\beta)=(0:1)$ and $(\alpha:\beta)=(1:-1)$. 

\begin{proposition}\label{prop:case01}
For $(\alpha:\beta)=(0:1)$, we have $\calP_{\alpha,\beta}(n)\cong V_n\oplus V_{n-1,1}$, and these modules are generated by linearizations of $a^{(n-1)}a^{n-1}$ and $a^{(n-1)}ba^{n-2}-b^{(n-1)}a^{n-1}$, respectively.
\end{proposition}

\begin{proof}In this case, Equation \eqref{eq:mat1} becomes
\[
\begin{pmatrix}
0&0& \beta\\
0&2\beta&\beta\\
0&2\beta&3\beta
\end{pmatrix}
\begin{pmatrix}
a'''a^3\\
a''a'a^2\\
(a')^3a
\end{pmatrix}=0,
 \]
implying $a''a'a^2=(a')^3a=0$. Partial multilinearizations of these identities are the identities
\begin{gather}
b''a'a^2+a''b'a^2+2a''a'ab=0,\label{eq:A90}\\
3(a')^2b'a+(a')^3b=0.\label{eq:A91}
\end{gather}
In equation \eqref{eq:mat2}, we should suppress the third row of the matrix since it corresponds to the identity \eqref{eq:A18} that we no longer have, so we get
 \[
\begin{pmatrix}
0&0&2&1&1&0&0\\
0&0&0&0&0&1&3\\
0&0&0&0&0&0&\beta\\
0&0&2\beta&0&0&\beta&0\\
0&0&2\beta&0&0&\beta&2\beta\\
0&0&0&0&2\beta&0&3\beta
\end{pmatrix}
\begin{pmatrix}
a'''a^2b\\
b'''a^3\\
a''a'ab\\
a''b'a^2\\
b''a'a^2\\
(a')^3b\\
b'(a')^2a
\end{pmatrix}=0,
 \]
easily implying $b'(a')^2a=(a')^3b=a''a'ab=b''a'a^2=a''b'a^2=0$. From $(a')^3a=(a')^3b=b'(a')^2a=0$, one immediately deduces $a'b'c'd=0$, since it is symmetric in $a,b,c$, and hence follows from its versions where at most two letters are different. 

Obtaining Equation \eqref{eq:mat3} did not use any relations using third derivatives (exactly the ones that we do not have), and the determinant of the corresponding matrix is equal to $-4\beta\ne 0$, so we obtain as before $a''a'b^2=a''b'ab=b''a'ab=b''b'a^2=0$. Furthermore, obtaining Equation \eqref{eq:mat4} did not use any relations using third derivatives (exactly the ones that we do not have), and the determinant of the corresponding matrix is equal to $4\beta\ne 0$, so we obtain as before 
 \[
a''b'ac=
a''c'ab=
b''a'ac=
c''a'ab=
b''c'a^2=
c''b'a^2=0,
 \] 
and hence also $a''b'cd=0$, since it is symmetric in $c,d$, and hence follows from its versions where at most three letters are different. 

\begin{lemma}\label{lm:21conseq}
In any Novikov algebra, the identity $a''b'cd=0$ implies 
 \[
a_1^{(k_1)}a_2^{(k_2)}\cdots a_n^{(k_n)}=0
 \]
for all $k_1+\cdots +k_n=n-1$, $k_1\geq \ldots \geq k_n\geq 0$ with $k_1\geq 2$, $k_2\geq 1$. 
\end{lemma}
\begin{proof}
We  prove this statement by induction on $n\ge 4$. The basis of induction is the identity $a''b'cd=0$ that we have. To prove the step of induction, we argue as follows. Assume that all such monomials of arity strictly less than $n\geq 5$ vanish, and consider a monomial $u=a_1^{(k_1)}a_2^{(k_2)}\cdots a_n^{(k_n)}$ of arity $n$. Since $k_1+\cdots +k_n=n-1$, $k_1\geq \ldots \geq k_n\geq 0$, we have $k_n=0$. Let us choose the maximal $p$ such that $k_p>0$, and complete the argument by induction on $k_p$. If $k_p=1$, then we can write $u$ as the product of the Novikov monomial $a_1^{(k_1)}a_2^{(k_2)}\cdots a_{i-1}^{(k_{i-1})}a_{i+1}^{(k_{i+1})} \cdots a_n^{(k_n)}$ and $a_i'$, and use the induction hypothesis for smaller $n$. Suppose that $k_p\geq2$. Then $a_1^{(k_1)}\cdots a_i^{(k_i-1)} \cdots a_{n-1}^{(k_{n-1})}$ is a Novikov monomial that vanishes by the induction hypothesis on $n$. Let us substitute $a_p:= a'_pa_n$ into that monomial. We obtain
 \[
a_1^{(k_1)}\cdots a_p^{(k_p)}\cdots a_{n}^{(k_n)}+
\left( \sum_{s=1}^{k_p-1} \binom{k_p-1}{s} a_p^{(k_p-s)} a_{n}^{(s)}\right) a_1^{(k_1)}\cdots a_{i-1}^{(k_{i-1})} a_{i+1}^{(k_{i+1})}\cdots a_{n-1}^{(k_{n-1})}
 \]
Note that for all $s=1,\ldots,k_p-1$ we have $\max(k_p-s,s)<k_p$, so the induction hypothesis applies to those terms, and $a_1^{(k_1)}\cdots a_p^{(k_p)}\cdots a_{n}^{(k_n)}$ vanishes, as needed.
\end{proof}

According to Lemma \ref{lm:21conseq}, we see that if a differential Novikov monomial $a_1^{(k_1)}\cdots a_n^{(k_n)}$ with $n\ge 4$ and $k_1\geq\ldots\geq k_n$ is \emph{a priori} nonzero in $\calP_{0,1}$, then either $k_1=\cdots=k_{n-1}=1$, $k_n=0$ or $k_1=n-1$, $k_2=\cdots=k_n=0$. However, we also have $a'b'c'd=0$, which immediately implies that $a_1'\cdots a_{n-1}'a_n=0$. Overall, this shows that $\calP_{1,-1}(n)$ is spanned by cosets of the $S_n$-orbit of $a_1^{(n-1)}a_2\cdots a_n$, so its dimension is at most $n$. To show that it is exactly $n$, note that the quotient of the operad $\calP_{0,1}(n)$ by the relation $a(bc)=0$ is the operad whose component of arity $n$ is clearly $n$-dimensional (it is the Koszul dual of the operad $\NAP$ of (right) non-associative permutative algebras \cite{MR2244257}). 
\end{proof}

\begin{proposition}\label{prop:case1-1}
For $(\alpha\colon\beta)=(1:-1)$, we have $\calP_{\alpha,\beta}(n)\cong V_n\oplus V_{n-1,1}$, and these modules are generated by   linearizations of $(a')^{n-1}a$ and $(a')^{n-1}b-b'(a')^{n-2}a$, respectively.
\end{proposition}

\begin{proof}
In this case, Equation \eqref{eq:mat1} becomes
 \[
\begin{pmatrix}
0&1&0\\
1&2&0\\
1&5&0
\end{pmatrix}
\begin{pmatrix}
a'''a^3\\
a''a'a^2\\
(a')^3a
\end{pmatrix}=0,
 \]
implying $a'''a^3=a''a'a^2=0$. Partial multilinearizations of these identities are the identities
\begin{gather*}
b'''a^3+3a'''a^2b=0,\\
b''a'a^2+a''b'a^2+2a''a'ab=0.
\end{gather*}
In equation \eqref{eq:mat2}, we should suppress the second row of the matrix since it corresponds to the identity \eqref{eq:A17} that we no longer have, so we get
 \[
\begin{pmatrix}
0&0&2&1&1&0&0\\
0&0&0&0&0&1&3\\
3&1&0&0&0&0&0\\
0&0&0&1&0&0&0\\
1&0&2&0&0&0&0\\
1&0&2&2&1&0&0\\
0&1&0&3&2&0&0
\end{pmatrix}
\begin{pmatrix}
a'''a^2b\\
b'''a^3\\
a''a'ab\\
a''b'a^2\\
b''a'a^2\\
(a')^3b\\
b'(a')^2a
\end{pmatrix}=0,
 \]
easily implying $a'''a^2b=b'''a^3=a''a'ab=a''b'a^2=b''a'a^2=0$. From $a'''a^3=a'''a^2b=b'''a^3=0$ one immediately deduces $a'''bcd=0$, since it is symmetric in $b,c,d$, and hence follows from its versions where at most two letters are different.

Moreover, even though Equations \eqref{eq:mat3} and \eqref{eq:mat4} were obtained using the equation $a'b'c'd=0$ which we no longer have, one notices that this equation is always used with the coefficient $\alpha+\beta$ which vanishes in our case. Since the determinants of the corresponding matrices are proportional to $\alpha-\beta$, they do not vanish, and we obtain as before $a''b'cd=0$. According to Lemma \ref{lm:21conseq}, we see that if a differential Novikov monomial $a_1^{(k_1)}\cdots a_n^{(k_n)}$ with $n\ge 4$ and $k_1\geq\ldots\geq k_n$ is \emph{a priori} nonzero in $\calP_{1,-1}$, then either $k_1=\cdots=k_{n-1}=1$, $k_n=0$ or $k_1=n-1$, $k_2=\cdots=k_n=0$. However, we also have $a'''bcd=0$, which implies by an easy induction using the result of Lemma \ref{lm:21conseq} that $a_1^{(n)}a_2\cdots a_n=0$. Overall, this shows that $\calP_{1,-1}(n)$ is spanned by cosets of the $S_n$-orbit of $a_1'\cdots a_{n-1}'a_n$, so its dimension is at most $n$. To show that it is exactly $n$, note that the quotient of the operad $\calP_{1,-1}(n)$ by the relation $(a,b,c)=0$ is the (left) associative permutative operad $\Perm$ whose component of arity $n$ is well known to be $n$-dimensional. 
\end{proof}
\end{proof}

\subsection{Quotienting out a copy of the two-dimensional module}\label{sec:2dim}

Let $(\gamma:\delta)\in\mathbb{P}^1$. We consider the operad $\calQ_{\gamma,\delta}$ that is the quotient of the Novikov operad by the ideal generated by the identity 
 \[
\gamma((a,a,b)-(b,a,a))+\delta(a(ab)-a(ba))=0.
 \]
In the differential realization, this is the identity
\begin{equation}\label{eq:4}
\gamma (a''ab-b''a^2)+\delta ((a')^2b-a'b'a)=0
\end{equation}
We shall now determine how the $S_n$-module structure of $\calQ_{\gamma,\delta}(n)$ depends on $(\gamma:\delta)$. Clearly, for all $(\gamma:\delta)\in\mathbb{P}^1$, we have 
 \[
\calQ_{\gamma,\delta}(1)\cong V_{1}, \quad \calQ_{\gamma,\delta}(2)\cong V_{2}\oplus V_{1,1}, \quad \calQ_{\gamma,\delta}(3)\cong V_{3}^2\oplus V_{2,1}.
 \]
We shall now prove the following theorem.

\begin{theorem}\label{th:rel2}
For all $(\gamma:\delta)\in\mathbb{P}^1$ and all $n\ge 4$, we have a $S_n$-module isomorphism 
 \[
\calQ_{\gamma,\delta}(n)\cong V_n^2\oplus V_{n-1,1}.
 \]
Moreover,  
\begin{itemize}
\item for $(\gamma:\delta)\ne(0:1)$, these modules are generated by linearizations of $(a')^{n-1}a$, $a''(a')^{n-3}a^2$, and $(a')^{n-1}b-b'(a')^{n-2}a$, respectively, 
\item for $(\gamma:\delta)=(0:1)$, these modules are generated by linearizations of $(a')^{n-1}a$, $a^{(n-1)}a^{n-1}$, and $a^{(n-1)}a^{n-2}b-b^{(n-1)}a^{n-1}$, respectively. 
\end{itemize}
\end{theorem}

\begin{proof}
Let us first consider a Novikov algebra satisfying the polynomial identity $\gamma (a''ab-b''a^2)+\delta ((a')^2b-a'b'a)=0$ with $(\gamma:\delta)\ne(0:1)$. Without loss of generality, we shall assume that $\gamma=1$ and work with the identity
\begin{equation}\label{eq:I}
a''ab-b''a^2+\delta ((a')^2b-a'b'a)=0.
\end{equation}

\begin{lemma}\label{lm:symmetryI}
If $(\gamma:\delta)\ne(0:1)$, we have 
\begin{equation}
a''b'cd-a''c'bd=0.\label{eq:I62}
\end{equation}
\end{lemma}

Proof of this lemma, once put in a human-readable form, is slightly technical, and is deferred to Appendix \ref{app:symmetryI}.

\begin{lemma}\label{lem3}
Suppose that $k_1+\cdots +k_n=n-1$, $k_1\geq \ldots \geq k_n\geq 0$ and $k_1\geq 2$, $k_2\geq 2$. If $(\gamma:\delta)\ne(0:1)$, we have
 \[
a_1^{(k_1)}a_2^{(k_2)}\cdots a_n^{(k_n)}=0.
 \]
\end{lemma}
\begin{proof}
We  prove this statement by induction on $n\ge 5$. The basis of induction is proved in Lemma \ref{lm:symmetryI}. To prove the step of induction, we argue as follows. Assume that all such monomials of arity strictly less than $n\geq 6$ vanish, and consider a monomial $u=a_1^{(k_1)}a_2^{(k_2)}\cdots a_n^{(k_n)}$ of arity $n$. Since $k_1+\cdots +k_n=n-1$, $k_1\geq \ldots \geq k_n\geq 0$, we have $k_n=0$. Let us choose the maximal $p$ such that $k_p>0$, and complete the argument by induction on $k_p$. If $k_p=1$, then we can write $u$ as the product of the Novikov monomial $a_1^{(k_1)}a_2^{(k_2)}\cdots a_{p-1}^{(k_{p-1})}a_{p+1}^{(k_{p+1})} \cdots a_n^{(k_n)}$ and $a_p'$, and use the induction hypothesis for smaller $n$. Suppose that $k_p\geq2$. Then $a_1^{(k_1)}\cdots a_p^{(k_p-1)} \cdots a_{n-1}^{(k_{n-1})}$ is a Novikov monomial that vanishes by the induction hypothesis on $n$. Let us substitute $a_p:= a'_pa_n$ into that monomial. We obtain
 \[
a_1^{(k_1)}\cdots a_p^{(k_p)}\cdots a_{n}^{(k_n)}+
\left( \sum_{s=1}^{k_p-1} \binom{k_p-1}{s} a_p^{(k_p-s)} a_{n}^{(s)}\right) a_1^{(k_1)}\cdots a_{p-1}^{(k_{p-1})} a_{p+1}^{(k_{p+1})}\cdots a_{n-1}^{(k_{n-1})}
 \]
Note that for all $s=1,\ldots,k_p-1$ we have $\max(k_p-s,s)<k_p$, so the induction hypothesis applies to those terms, and $a_1^{(k_1)}\cdots a_p^{(k_p)}\cdots a_{n}^{(k_n)}$ vanishes, as needed.
\end{proof}

This lemma already implies that for $\gamma\ne 0$, the component $\calQ_{\gamma,\delta}(n)$ with $n\ge 4$ is a sum of $S_n$-submodules spanned by orbits of $a_1^{(k)}a_2'\cdots a_{n-k}'a_{n-k+1}\cdots a_n$ with $1\le k\le n-1$. Let us show that each such submodule is at most $n$-dimensional. For that, we prove the following lemma. 

\begin{lemma}\label{lem5}
If $(\gamma:\delta)\ne(0:1)$, then for all $n\geq 4$, we have
$$
a_{1}^{(n-2)}a'_{2}a_3a_4\cdots a_n-a_{1}^{(n-2)}a_{2}a'_{3}a_4\cdots a_n=0.
$$
\end{lemma}
\begin{proof}
Induction on $n$. The basis of induction is proved in Lemma \ref{lm:symmetryI}. Assume that 
\begin{equation}\label{eq:I65}
u=a_1^{(n-2)}a'_2a_3\cdots a_n-a_1^{(n-2)}a_2a'_3a_4\cdots a_n=0.
\end{equation}
We have 
\begin{multline*}
u'a_{n+1}-\sum_{i=4}^n u(a_1,\ldots,a_{i-1},a_{n+1},a_{i+1},\ldots,a_n)a_i'=\\
a_1^{(n-1)}a'_2a_3\cdots a_na_{n+1}-a_1^{(n-1)}a_2a'_3a_4\cdots a_na_{n+1}+\\
a_1^{(n-2)}a''_2a_3\cdots a_na_{n+1}-a_1^{(n-2)}a_2a''_3a_4\cdots a_na_{n+1},
\end{multline*}
and the last two terms vanish thanks to Lemma \ref{lem3}, proving the step of induction.
\end{proof}

Since we can multiply each such relation by an arbitrary number of derivatives, it follows that the $S_n$-submodule spanned by the orbit of $a_1^{(k)}a_2'\cdots a_{n-k}'a_{n-k+1}\cdots a_n$ is spanned as a vector space by the $n$ elements of the orbit where the $k$-th derivative is applied to $a_i$, $1\leq i\leq n$. In the Novikov operad, each such module contains a copy of $V_n$ and a copy of $V_{n-1,1}$. 

\begin{lemma}\label{lem6}
If $(\gamma:\delta)\ne(0:1)$, then for all $n\geq 4$, we have
 \[
a^{(n-1)}a^{n-2}b+(n-2+\delta)a^{(n-2)}a^{n-2}b'=0.
 \]
\end{lemma}

\begin{proof}
Induction on $n$. Let us establish the basis of induction, which is the equation
\begin{equation}\label{eq:I16}
a'''a^2b+(2 +\delta)a''b'a^2=0.
\end{equation}

Substituting $b:=a'b$ into \eqref{eq:I}, we obtain 
 \[
a''a'ab-a'''a^2b-2a''b'a^2-b''a'a^2+\delta ((a')^3b-(a')^2b'a-a''a'ab)=0.
 \]
Multiplying \eqref{eq:I} by $a'$, we obtain 
 \[
a''a'ab-b''a'a^2+\delta ((a')^3b-(a')^2b'a)=0.
 \]
Subtracting these two, we obtain
\begin{equation}\label{eq:I15}
a'''a^2b+2 a''b'a^2+\delta a''a'ab=0,
\end{equation}
which, thanks to Lemma \ref{lm:sym1aaab}, implies \eqref{eq:I16}.

Let us show how to prove the step of induction. Suppose that
\begin{equation}\label{eq:I67}
a^{(n-1)}a^{n-2}b+(n-2+\delta)a^{(n-2)}b'a^{n-2}=0.
\end{equation}
Taking the derivative of \eqref{eq:I67} and multiplying by $a$ we obtain
\begin{multline*}
a^{(n)}a^{n-1}b+(n-2)a^{(n-1)}a'a^{n-2}b+a^{(n-1)}a^{n-1}b'+\\
(n-2+\delta) (a^{(n-1)}b'a^{n-1}+a^{(n-2)}b''a^{n-1}+(n-2) a^{(n-2)}b'a'a^{n-2})=0. 
\end{multline*}
Since $a^{(n-2)}b''a^{n-1}=0$ because of Lemma \ref{lem3} and 
 \[
(n-2)a^{(n-1)}a'a^{n-2}b+(n-2+\delta)(n-2) a^{(n-2)}b'a'a^{n-2}=0
 \] 
because of the induction hypothesis, we have 
 \[
a^{(n)}a^{n-1}b+(n-1+\delta)a^{(n-1)}b'a^{n-1}=0,
 \]
as needed. 
\end{proof}

Using Lemma \ref{lem5}, we can rewrite the result of Lemma \ref{lem6} as
 \[
a^{(n-1)}a^{n-2}b+(n-2+\delta)a^{(n-2)}a^{n-3}ba'=0,
 \]
or, otherwise speaking, 
 \[
a^{(n-1)}b=-(n-2+\delta)a^{(n-2)}ba^{n-3}a'.
 \]
This equation can be iterated, obtaining  
\begin{multline*}
a^{(n-1)}ba^{n-2}=(n-2+\delta)(n-3+\delta)a^{(n-3)}ba^{n-4}(a')^2=\\
-(n-2+\delta)(n-3+\delta)(n-4+\delta)a^{(n-4)}ba^{n-5}(a')^3=\ldots\\
\ldots=(-1)^{n-3}\prod_{k=2}^{n-2}(n-k+\delta)a^{(2)}ba(a')^{n-3}.
\end{multline*}
This implies that at most two copies of $V_{n-1,1}$ and at most two different copies of $V_{n}$ survive in our quotient. In fact, only one copy of $V_{n-1,1}$ survives. This is true since Equation \eqref{eq:I} implies that
 \[
a^{(2)}ba(a')^{n-3}-b^{(2)}a^2(a')^{n-3}+\delta((a')^{n-1}b-(a')^{n-2}b'a)=0,
 \]
literally relating those two copies.

Finally, let us show that both copies of $V_n$ (generated by linearizations of $(a')^{n-1}a$ and $a''(a')^{n-3}a^2$), and the copy of $V_{n-1,1}$ (generated by $(a')^{n-1}b-b'(a')^{n-2}a$) survive in the quotient. For that, we shall consider two particular Novikov algebras. The first one, $A$, is the one-dimensional simple Novikov algebra: it has a basis $e$ such that $ee=e$; clearly, this algebra belongs to the variety we consider. The second one is the algebra $B_\delta$ with a basis $e,f$ and the multiplication table
 \[
ee=0, \quad ef=-\delta e,\quad
fe=e, \quad ff=f.
 \]
It is proved in Proposition \ref{prop:Bdelta} that $B_\delta$ is a Novikov algebra that belongs to the variety we consider. We note that the identity $(a')^{n-1}b-b'(a')^{n-2}a=0$ corresponds to the identity $a (a\cdots (ab)\cdots)=b(a(a\cdots aa^2\cdots))$ in Novikov algebras. If we consider the algebra $B_\delta$ and set $a=f$, $b=e$, we get $e=-\delta e$, which is false for $\delta\ne -1$, so the submodule $V_{n-1,1}$ survives in the quotient. 
Suppose we have $\lambda (a')^{n-1}a+\mu a''(a')^{n-3}a^2=0$ in the quotient. We note that this identity corresponds to the identity
\begin{equation}\label{eq:suspect}
\lambda a(a\cdots aa^2)+\mu a(a\cdots (a(a,a,a))\cdots)=0
\end{equation}
in Novikov algebras. If we consider the algebra $A$ and set $a=e$, Equation \eqref{eq:suspect} becomes $\lambda e=0$, so $\lambda=0$. In the algebra $B_\delta$, we have 
\begin{gather*}
(e+f)(e+f)=(-\delta+1)e+f,\\
(e+f)((-\delta+1)e+f)=(-\delta-\delta+1)e+f=(-2\delta+1)e+f,\\
(e+f)((-2\delta+1)e+f)=(-\delta-2\delta+1)e+f=(-3\delta+1)e+f,
\end{gather*}
and by induction we easily obtain $L_{e+f}^n(e+f)=(-n\delta+1)e+f$. On the other hand, 
\begin{gather*}
((-\delta+1)e+f)(e+f)=(-\delta(-\delta+1)+1)e+f=(\delta^2-\delta+1)e+f,\\
(e+f,e+f,e+f)=(\delta^2+\delta)e,\\
(e+f)(e+f,e+f,e+f)=(\delta^2+\delta)(e+f)e=(\delta^2+\delta)e,
\end{gather*}
and by induction we easily obtain $R_{e+f}^k(e+f,e+f,e+f)=(\delta^2+\delta)e$.
Thus, if we consider the algebra $B_\delta$ and set $a=e+f$, Equation \eqref{eq:suspect} becomes $\lambda((-n\delta+1)e+f)+\mu (\delta^2+\delta)e=0$. Since we already established that $\lambda=0$, this implies $\mu=0$ for  $\delta\notin\{0,-1\}$.

The cases $(\gamma:\delta)=(1:-1)$ and $(\gamma:\delta)=(1:0)$ need to be considered separately. The linear independence of the corresponding modules follows from Corollary \ref{cor:dim} established using operadic Gr\"obner bases.

Let us now consider the case $(\gamma:\delta)=(0:1)$, that is the identity
\begin{equation}\label{eq:II}
(a')^2b-a'b'a=0.
\end{equation}

\begin{lemma}\label{lm:symmetryII}
If $(\gamma:\delta)=(0:1)$, we have 
\begin{equation}
a''b'cd=0.\label{eq:IIa}
\end{equation}
\end{lemma}

Proof of this lemma, once put in a human-readable form, is slightly technical, and is deferred to Appendix \ref{app:symmetryII}.

Using Lemmas \ref{lm:symmetryII} and \ref{lm:21conseq}, we conclude that for $(\gamma:\delta)=(0:1)$, we have  
 \[
a_1^{(k_1)}a_2^{(k_2)}\cdots a_n^{(k_n)}=0
 \]
whenever $k_1+\cdots +k_n=n-1$, $k_1\geq \ldots \geq k_n\geq 0$ and $k_1\geq 2$, $k_2\geq 1$. This means that for $n\ge 4$, it is enough to consider the submodules generated by $S_n$-orbits of $a_1^{(n-1)}a_2\cdots a_n$ and $a_1'\cdots a_{n-1}'a_n$, each at most $n$-dimensional. Moreover, multiplying \eqref{eq:II} by several copies of $a'$, we obtain $(a')^{n-1}b-(a')^{n-2}b'a$, proving that the copy of $V_{n-1,1}$ in the second of these submodule vanishes. Finally, to show that both copies of $V_n$ (generated by linearizations of $(a')^{n-1}a$ and $a^{(n-1)}a^{n-1}$), and the copy of $V_{n-1,1}$ (generated by $a^{(n-1)}a^{n-2}b-b^{(n-1)}a^{n-1}$) survive in the quotient, one can use Corollary \ref{cor:dim} established using operadic Gr\"obner bases. 
\end{proof}

\subsection{Combining the two identities}\label{sec:both}

For $\rho=((\alpha:\beta),(\gamma:\delta))\in\mathbb{P}^1\times\mathbb{P}^1$, let us denote by $\calO_\rho$ the quotient of the operad of Novikov algebras by the ideal generated by the identities
\begin{gather*}
\alpha a''a^2+(\alpha+\beta) (a')^2a =0,\\
\gamma (a''ab-b''a^2)+\delta ((a')^2b-a'b'a)=0.
\end{gather*}
We shall now determine how the $S_n$-module structure of $\calO_{\rho}(n)$ depends on $\rho$. Clearly, for all $\rho\in\mathbb{P}^1\times\mathbb{P}^1$, we have 
 \[
\calO_\rho(1)\cong V_{1},\quad \calO_\rho(2)\cong V_{2}\oplus V_{1,1},\quad \calO_\rho(3)\cong V_{3}\oplus V_{2,1}.
 \]

We shall now use Theorems \ref{th:rel1} and \ref{th:rel2} together and prove the following result.

\begin{theorem}\label{th:rel12}
Let $\rho\in\mathbb{P}^1\times\mathbb{P}^1$, and let $n\ge 4$. The $S_n$-module structure of $\calO_{\rho}(n)$ is described as follows:
\begin{itemize}

\item for $\rho=((0:1),(0:1))$, we have $\calO_\rho(n)=V_n\oplus V_{n-1,1}$. Moreover, $\calO_\rho\cong \NAP^!$,

\item for $\rho=((1:-1),(\gamma:\delta))$ with $\delta\ne 0$, we have $\calO_\rho(n)=V_n$, 

\item for $\rho=((1:-1),(1:0))$, we have $\calO_\rho(n)=V_n\oplus V_{n-1,1}$. Moreover, $\calO_\rho\cong \Perm$,

\item in all other cases, we have $\calO_\rho(n)=0$.
\end{itemize}
\end{theorem}

\begin{proof}

Suppose that $\rho=((0:1),(0:1))$, so that our identities are $aa^2=0$ and $a(ab)-ba^2=0$. These imply $a(bc)=0$, and hence the discussion in the proof of Theorem \ref{th:rel1} implies $\calO_\rho\cong \NAP^!$. 

Suppose that $\rho=((1:-1),(1:0))$, so that our identities are $(a,a,a)=0$ and $(a,a,b)-(b,a,a)=0$. These imply $(a,b,c)=0$, and hence the discussion in the proof of Theorem \ref{th:rel1} implies $\calO_\rho\cong \Perm$. 

Suppose that $\rho=((1:-1),(\gamma:\delta))$ with $\delta\ne 0$. We already know from the proof of Theorem \ref{th:rel1} that the module $\calP_{1,-1}(n)$ is  spanned by cosets of the $S_n$-orbit of $a_1'\cdots a_{n-1}'a_n$, and that for $n\ge 4$ all other orbits of differential Novikov monomials vanish in $\calP_{1,-1}(n)$. Now, if we multiply 
 \[
\gamma (a''ab-b''a^2)+\delta ((a')^2b-a'b'a)=0
 \]
by $a'$ and simplify using the abovementioned vanishing, we obtain, since $\delta\ne 0$,
 \[
(a')^3b-(a')^2b'a=0,
 \] 
which implies that the quotient is at most one-dimensional (and spanned by the linearization of $(a')^{n-1}a$). It is clear that the one-dimensional module survives in the quotient, since in this case the operad $\Com$ of commutative algebras is a quotient of $\calO_\rho$. 

Let us complete the proof by showing that in all other cases we have $\calO_\rho(n)=0$ for $n\ge 4$.

Suppose that $\rho=((1:1),(\gamma:\delta))$. We know from Theorem \ref{th:rel1} that $\calP_{\alpha,\beta}(n)=0$ for $n\ge 6$, and also that 
\begin{gather*}
\calP_{\alpha,\beta}(4)\cong V_{3,1}\oplus V_{2,2}\oplus V_{2,1,1},\\
\calP_{\alpha,\beta}(5)\cong V_{2,2,1}.
\end{gather*}
At the same time, we know from Theorem \ref{th:rel2} that $\calQ_{\gamma,\delta}(n)\cong V_n\oplus V_{n-1,1}$. Since $\calO_\rho$ is a common quotient of these two operads, we clearly have $\calO_\rho(5)=0$. In arity $4$, we may \emph{a priori} have a copy of $V_{3,1}$. However, we know from Theorem \ref{th:rel1} that the submodule of $\calP_{\alpha,\beta}(4)$ isomorphic to $V_{3,1}$ is spanned by linearizations of $a''a'ab-a''b'a^2$. However, for $\gamma\ne 0$, according to Lemma \ref{lm:symmetryI}, we have $a''b'cd-a''c'bd=0$, implying $a''a'ab-a''b'a^2=0$, and for $\gamma=0$, according to Lemma \ref{lm:symmetryII}, we have $a''b'cd=0$, implying $a''a'ab-a''b'a^2=0$.

Suppose that $\rho=((0:1),(\gamma:\delta))$ with $\gamma\ne 0$. We already know from the proof of Theorem \ref{th:rel1} that $a''a'a^2=a'b'c'd=0$ in $\calP_{\alpha,\beta}$. At the same time, we know from Theorem \ref{th:rel2} that for $\gamma\ne 0$, the module $\calQ_{\gamma,\delta}(n)$ is spanned by linearizations of $(a')^{n-1}a$, $a''(a')^{n-3}a^2$, and $(a')^{n-1}b-b'(a')^{n-2}a$. It follows that all these linearizations vanish in $\calO_\rho(n)$.

Finally, if $\rho=((\alpha:\beta),(\gamma:\delta))$ with $\alpha(\alpha-\beta)(\alpha+\beta)\ne 0$, we have $\calP_{\alpha,\beta}(n)=0$, so the same is true for the component $\calO_\rho(n)=0$ of the quotient operad $\calO_\rho$. 
\end{proof}

\subsection{The main theorem}\label{sec:mainth}

We are now able to prove the main result of this article.

\begin{theorem}\label{th:maindistr}
The lattice of subvarieties of a variety of Novikov algebras is distributive if and only if all algebras of that variety satisfy the identities
 \[
\alpha a^2a+\beta aa^2,\quad
\gamma((a,a,b)-(b,a,a))+\delta(a(ab)-ba^2)
 \]
for some $((\alpha:\beta),(\gamma:\delta))\in\mathbb{P}^1\times\mathbb{P}^1$.
\end{theorem}

\begin{proof}
First of all, since the arity $3$ component of the Novikov operad is, as a $S_3$-module, isomorphic to $V_3^2\oplus V_{2,1}^2$, Proposition \ref{prop:distr} ensures that identities of this form are satisfied in $\mathfrak{M}$. Moreover, if those identities are satisfied, Theorem \ref{th:rel12} implies that the corresponding lattice is distributive. 
\end{proof}

For completeness, let us describe the corresponding distributive lattices. We shall do it by displaying the diagrams that indicate which of the modules of identities imply the other ones. 

If $\rho=((0:1),(0:1))$, the corresponding diagram is 
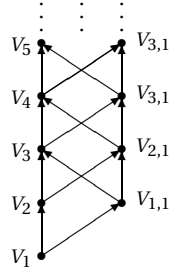
\vspace{0.7cm}
 \[
\begin{picture}(20,80)
\put(0,0){\circle*{3}} \put(0,20){\circle*{3}}
\put(0,40){\circle*{3}} \put(0,60){\circle*{3}}
\put(0,80){\circle*{3}}
\put(0,85){\circle*{1}} \put(0,90){\circle*{1}}
\put(0,95){\circle*{1}} 
\put(15,85){\circle*{1}} \put(15,90){\circle*{1}}
\put(15,95){\circle*{1}}
\put(30,20){\circle*{3}} \put(30,40){\circle*{3}}
\put(30,60){\circle*{3}} \put(30,80){\circle*{3}}
\put(30,85){\circle*{1}} \put(30,90){\circle*{1}}
\put(30,95){\circle*{1}}
\put(0,0){\vector(0,2){20}}\put(-12,-3){\tiny{$V_1$}}
\put(0,20){\vector(0,2){20}}\put(-12,17){\tiny{$V_2$}}
\put(0,40){\vector(0,2){20}}\put(-12,37){\tiny{$V_3$}}
\put(-12,57){\tiny{$V_4$}} 
\put(0,60){\vector(0,2){20}} \put(-12,77){\tiny{$V_5$}}
\put(0,0){\vector(3,2){30}}\put(35,19){\tiny{$V_{1,1}$}}
\put(30,20){\vector(0,2){20}}\put(35,39){\tiny{$V_{2,1}$}}
\put(30,40){\vector(0,2){20}}
\put(35,59){\tiny{$V_{3,1}$}}
\put(30,60){\vector(0,2){20}}
\put(35,79){\tiny{$V_{3,1}$}}
\put(0,20){\vector(3,2){30}}  
\put(0,40){\vector(3,2){30}} 
\put(30,20){\vector(-3,2){30}}
\put(30,40){\vector(-3,2){30}}
\put(0,60){\vector(3,2){30}} \put(30,60){\vector(-3,2){30}}
\put(0,60){\vector(3,2){30}} 
\end{picture}
 \]
The fact that either $V_n$ or $V_{n-1,1}$ implies $V_{n+1}\oplus V_{n,1}$ follows from the observation that if we substitute $a_1:=a_1'a_{n+1}$ instead of $a_1$ in the linearization of either the symmetrization of $a_1^{(n-1)}a_2\cdots a_n$ or $a_1^{(n-1)}a_2\cdots a_n-a_2^{(n-1)}a_1\cdot a_n$, the result clearly generates $\calO_\rho(n+1)$ as an $S_{n+1}$-module, since we know that most differential Novikov monomials vanish, and the result is proportional to $a_1^{(n)}a_2\cdots a_na_{n+1}$. 

If $\rho=((1:-1),(1:0))$, the corresponding diagram is 
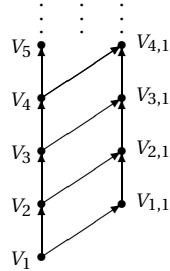
\vspace{0.7cm}
 \[
\begin{picture}(20,80)
\put(0,0){\circle*{3}} \put(0,20){\circle*{3}}
\put(0,40){\circle*{3}} \put(0,60){\circle*{3}}
\put(0,80){\circle*{3}}
\put(0,85){\circle*{1}} \put(0,90){\circle*{1}}
\put(0,95){\circle*{1}} 
\put(15,85){\circle*{1}} \put(15,90){\circle*{1}}
\put(15,95){\circle*{1}}

\put(30,20){\circle*{3}} \put(30,40){\circle*{3}}
\put(30,60){\circle*{3}} \put(30,80){\circle*{3}}
\put(30,85){\circle*{1}} \put(30,90){\circle*{1}}
\put(30,95){\circle*{1}}

\put(0,0){\vector(0,2){20}}\put(-12,-3){\tiny{$V_1$}}
\put(0,20){\vector(0,2){20}}\put(-12,17){\tiny{$V_2$}}
\put(0,40){\vector(0,2){20}}\put(-12,37){\tiny{$V_3$}}
\put(-12,57){\tiny{$V_4$}} 
\put(0,60){\vector(0,2){20}} \put(-12,77){\tiny{$V_5$}}

\put(0,0){\vector(3,2){30}}\put(35,19){\tiny{$V_{1,1}$}}
\put(30,20){\vector(0,2){20}}\put(35,39){\tiny{$V_{2,1}$}}
\put(30,40){\vector(0,2){20}}
\put(35,59){\tiny{$V_{3,1}$}}
\put(30,60){\vector(0,2){20}}
\put(35,79){\tiny{$V_{4,1}$}}

\put(0,20){\vector(3,2){30}}  
\put(0,40){\vector(3,2){30}} 
\put(0,60){\vector(3,2){30}}
\put(0,60){\vector(3,2){30}} 
\end{picture}
 \]
The fact that $V_n$ implies $V_{n+1}\oplus V_{n,1}$ follows from noting that, if we denote by $u$ the linearization of $a(a')^{n-1}$, the product $u'a_{n+1}$ clearly generates $\calO_\rho(n+1)$ as an $S_{n+1}$-module, since we know that most differential Novikov monomials vanish, and the result is proportional to $a_1'\cdots a_n'a_{n+1}$. The fact that $V_{n-1,1}$ implies $V_{n,1}$ is proved by a similar calculation. The fact that $V_{m,1}$ does not imply $V_n$ is clear from the fact that our operad admits the operad $\Com$ as a quotient. 

If $\rho=((1:1),(\gamma:\delta))$, the corresponding diagram is 
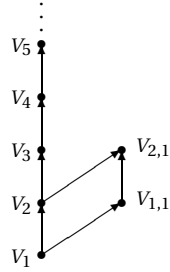
\vspace{0.7cm}
 \[
\begin{picture}(20,80)
\put(0,0){\circle*{3}} \put(0,20){\circle*{3}}
\put(0,40){\circle*{3}} \put(0,60){\circle*{3}}
\put(0,80){\circle*{3}}
\put(0,85){\circle*{1}} \put(0,90){\circle*{1}}
\put(0,95){\circle*{1}} 
\put(30,20){\circle*{3}} \put(30,40){\circle*{3}}
\put(0,0){\vector(0,2){20}}\put(-12,-3){\tiny{$V_1$}}
\put(0,20){\vector(0,2){20}}\put(-12,17){\tiny{$V_2$}}
\put(0,40){\vector(0,2){20}}\put(-12,37){\tiny{$V_3$}}
\put(-12,57){\tiny{$V_4$}} 
\put(0,60){\vector(0,2){20}} \put(-12,77){\tiny{$V_5$}}
\put(0,0){\vector(3,2){30}}\put(35,19){\tiny{$V_{1,1}$}}
\put(30,20){\vector(0,2){20}}\put(35,39){\tiny{$V_{2,1}$}}
\put(0,20){\vector(3,2){30}}  
\end{picture}
 \]
The fact that $V_{1,1}$ and $V_{2,1}$ do not imply $V_n$ is clear from the fact that our operad admits the operad $\Com$ as a quotient.

In all other cases the corresponding diagram is, of course,
 
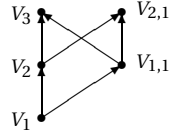
\[
\begin{picture}(20,50)
\put(0,0){\circle*{3}} \put(0,20){\circle*{3}}
\put(0,40){\circle*{3}} 
\put(30,20){\circle*{3}} \put(30,40){\circle*{3}}
\put(0,0){\vector(0,2){20}}\put(-12,-3){\tiny{$V_1$}}
\put(0,20){\vector(0,2){20}}\put(-12,17){\tiny{$V_2$}}
\put(-12,37){\tiny{$V_3$}}
\put(0,0){\vector(3,2){30}}\put(35,19){\tiny{$V_{1,1}$}}
\put(35,39){\tiny{$V_{2,1}$}}
\put(0,20){\vector(3,2){30}} 
\put(30,20){\vector(-3,2){30}} \put(30,20){\vector(0,2){20}}
\end{picture}
 \]

\section{Koszul operads with the Novikov operad as a quotient}\label{sec:koszulthm}

In this section, we study the Koszul property of quotients of the Novikov operad. Let us start by remarking that the elements $a^2a$ and $aa^2$ considered in Theorem \ref{th:rel1}, once multilinearized, become
\begin{gather*} 
(a_1a_2)a_3+(a_2a_3)a_1+(a_3a_1)a_2+(a_2a_1)a_3+(a_3a_2)a_1+(a_1a_3)a_2),\\
a_1(a_2a_3)+a_2(a_3a_1)+a_3(a_1a_2)+a_1(a_3a_2)+a_2(a_1a_3)+a_3(a_2a_1).
\end{gather*}
Because of the Novikov identities, we have 
\begin{multline*}
(a_2a_1)a_3+(a_3a_2)a_1+(a_1a_3)a_2=\\
(a_2,a_1,a_3)+a_2(a_1a_3)+(a_3,a_2,a_1)+a_3(a_2a_1)+(a_1,a_3,a_2)+a_1(a_3a_2)=\\
(a_2,a_3,a_1)+a_1(a_2a_3)+(a_3,a_1,a_2)+a_2(a_3a_1)+(a_1,a_2,a_3)+a_3(a_1a_2)=\\
(a_1a_2)a_3+(a_2a_3)a_1+(a_3a_1)a_2,
\end{multline*}
and 
 \[
a_1(a_3a_2)+a_2(a_1a_3)+a_3(a_2a_1)=a_1(a_2a_3)+a_2(a_3a_1)+a_3(a_1a_2),
 \]
and so in the Novikov operad the former elements are proportional to 
\begin{gather*} 
(a_1a_2)a_3+(a_2a_3)a_1+(a_3a_1)a_2,\\
a_1(a_2a_3)+a_2(a_3a_1)+a_3(a_1a_2).
\end{gather*}
Also, the multilinearizations of the elements $(a,a,b)-(b,a,a)$ and $a(ab)-a(ba)$ considered in Theorem \ref{th:rel12}, are
\begin{gather*}
(a_1,a_2,a_3)+(a_2,a_1,a_3)-(a_3,a_1,a_2)-(a_3,a_2,a_1),\\
a_1(a_2a_3)+a_2(a_1a_3)-a_3(a_1a_2)-a_3(a_2a_1),\\
\end{gather*}
which, modulo the Novikov identities, simplify to 
\begin{gather*}
(a_1,a_2,a_3)+(a_2,a_1,a_3)-2(a_3,a_1,a_2),\\
2a_1(a_2a_3)-a_3(a_1a_2)-a_3(a_2a_1).
\end{gather*}

Finally, we record the polarized presentation of the Novikov operad it is proved by a direct calculation, and we omit the proof).

\begin{proposition}
The polarized presentation of the Novikov operad exhibits it as the quotient by the ideal generated by the elements 
\begin{gather*}
[[a_1,a_2],a_3]+[[a_2,a_3],a_1]+[[a_3,a_1],a_2],\\
[a_1,a_2]\cdot a_3+[a_2,a_3]\cdot a_1+[a_3,a_1]\cdot a_2,\\
2[a_1,a_2]\cdot a_3-[a_1\cdot a_3,a_2]-[[a_1,a_2],a_3]-[a_1,a_2\cdot a_3], \\
(a_1\cdot a_2)\cdot a_3-a_1\cdot (a_2\cdot a_3)-[a_1,a_3]\cdot a_2
\end{gather*}
\end{proposition}

Since the kernel of the quotient map is generated by an $S_3$-submodule of the arity three component, there are nine cases to consider: the multiplicity of the trivial module may be equal to $0$, $1$, or $2$, and the multiplicity of the irreducible two-dimensional module may be equal to $0$, $1$, or $2$.
\begin{proposition}[Multiplicities $(0,0)$]\label{prop:first}
The quotient of the Novikov operad by the zero ideal is not Koszul. 
\end{proposition}

\begin{proof}
This is established by Dzhumadildaev \cite{MR2834140}.
\end{proof}

\begin{proposition}[Multiplicities $(1,0)$]\label{prop:second}
The quotient of the Novikov operad by the ideal generated by 
 \[
\alpha((a_1a_2)a_3+(a_2a_3)a_1+(a_3a_1)a_2)+\beta(a_1(a_2a_3)+a_2(a_3a_1)+a_3(a_1a_2))
 \]
is not Koszul for any $(\alpha:\beta)\in\mathbb{P}^1$. 
\end{proposition}

\begin{proof}
This is precisely the operad $\calP_{\alpha,\beta}$ studied in Theorem \ref{th:rel1}, where it is established that its Poincaré series is equal to 
\begin{itemize}
\item $t+t^2+\frac56t^3+\sum_{k\ge 4}\frac{t^k}{(k-1)}$ for $(\alpha:\beta)=(0:1)$ and $(\alpha\colon\beta)=(1:-1)$, and the compositional inverse of this series has a negative coefficient $-\frac{11}{24}$ at $t^5$,
\item $t+t^2+\frac56t^3+\frac{1}{3}t^4+\frac{1}{24}t^5$ for $(\alpha\colon\beta)=(1:1)$, and the compositional inverse of this series has a positive coefficient $\frac{35}{24}$ at $t^6$,
\item $t+t^2+\frac56t^3$ otherwise, and the compositional inverse of this series has a negative coefficient $-\frac{17}{12}$ at $t^5$.
\end{itemize}
In all of the above cases, by Proposition \ref{prop:positivity}, this operad is not Koszul.
\end{proof}

\begin{proposition}[Multiplicities $(2,0)$]\label{prop:third}
The quotient of the Novikov operad by the ideal generated by 
 \[
(a_1a_2)a_3+(a_2a_1)a_3+(a_3a_1)a_2,\qquad
a_1(a_2a_3)+a_2(a_3a_1)+a_3(a_1a_2)
 \]
is not Koszul.
\end{proposition}

\begin{proof}
If we take the quotient by both copies of the trivial representation, we can first take the quotient by the copy corresponding to $(\alpha:\beta)=(1:1)$, which is one-dimensional in each arity starting from $4$, and then by the remaining copy. It follows that our operad vanishes from the arity $4$ onwards, and its Poincaré series is $t+t^2+\frac{2}{3}t^3$. Its compositional inverse has a positive coefficient $\frac{14}9$ at $t^6$. By Proposition \ref{prop:positivity}, this operad is not Koszul.
\end{proof}

\begin{proposition}[Multiplicities $(0,1)$]\label{prop:fourth}
The quotient of the Novikov operad by the ideal generated by 
 \[
\gamma((a_1,a_2,a_3)+(a_2,a_1,a_3)-2(a_3,a_1,a_2))+\delta(2a_1(a_2a_3)-a_3(a_1a_2)-a_3(a_2a_1))
 \]
is not Koszul. 
\end{proposition}

\begin{proof}
This is the operad $\calQ_{\gamma,\delta}$ studied in Theorem \ref{th:rel2}, where it is established that for all $(\gamma:\delta)$, we have $\dim\calQ_{\gamma,\delta}(n)=n+1$, so the Poincaré series of our operad is given by 
 \[
t+t^2+\sum_{k\ge 3}\frac{(k+1)t^k}{k!}.
 \]
Its compositional inverse has a positive coefficient $\frac{13667}{5760}$ at $t^8$. By Proposition \ref{prop:positivity}, this operad is not Koszul.
\end{proof}

\begin{proposition}[Multiplicities $(1,1)$]\label{prop:fifth}
The quotient of the Novikov operad by the ideal generated by 
\begin{gather*}
\alpha((a_1a_2)a_3+(a_2a_1)a_3+(a_3a_1)a_2)+\beta(a_1(a_2a_3)+a_2(a_3a_1)+a_3(a_1a_2)),\\
\gamma((a_1,a_2,a_3)+(a_2,a_1,a_3)-2(a_3,a_1,a_2))+\delta(2a_1(a_2a_3)-a_3(a_1a_2)-a_3(a_2a_1))
\end{gather*}
is Koszul if and only if $((\alpha:\beta),(\gamma:\delta))$ is equal to $((0:1),(0:1))$ or $((1:-1),(1:0))$. 
\end{proposition}

\begin{proof}
This is the operad $\calO_{\rho}$ studied in Theorem \ref{th:rel12}, where it is established that for $((\alpha:\beta),(\gamma:\delta))=((0:1),(0:1))$ this operad is isomorphic to the operad $\NAP^!$, and for $((\alpha:\beta),(\gamma:\delta))=((1:-1),(1:0))$ this operad is isomorphic to the operad $\Perm$; both these operads are well known to be Koszul. For $((\alpha:\beta),(\gamma:\delta))=((1:-1),(\gamma:\delta))$ with $\delta\ne 0$, Theorem \ref{th:rel12} implies that the Poincaré series of this operad is
 \[
t+t^2+\frac12t^3+\sum_{k\ge 4}\frac{t^k}{k!}
 \] 
whose compositional inverse has a negative coefficient $-\frac{802543633}{39916800}$ at $t^{11}$. By Proposition \ref{prop:positivity}, this operad is not Koszul.
For all other cases, Theorem \ref{th:rel12} implies that we have $\calO_{\rho}(n)=0$ for $n\ge 4$, so the Poincaré series of this operad is $t+t^2+\frac12t^3$. Its compositional inverse has a positive coefficient $\frac{715}{16}$ at $t^{10}$. By Proposition \ref{prop:positivity}, this operad is not Koszul.
\end{proof}

\begin{proposition}[Multiplicities $(2,1)$]\label{prop:sixth}
The quotient of the Novikov operad by the ideal generated by 
\begin{gather*}
(a_1a_2)a_3+(a_2a_3)a_1+(a_3a_1)a_2,\qquad
a_1(a_2a_3)+a_2(a_3a_1)+a_3(a_1a_2),\\
\gamma((a_1,a_2,a_3)+(a_2,a_1,a_3)-2(a_3,a_1,a_2))+\delta(2a_1(a_2a_3)-a_3(a_1a_2)-a_3(a_2a_1))
\end{gather*}
is Koszul.
\end{proposition}

\begin{proof}
Let us denote this operad by $\calS_{\gamma,\delta}$. If we take the quotient by both copies of the trivial representation, we can first take the quotient by the one of its two copies that corresponds to $(\alpha:\beta)=(1:1)$, which is one-dimensional in each arity starting from $4$, and then by the remaining copy. It follows that our operad vanishes from the arity $4$ onwards, and its Poincaré series is $t+t^2+\frac{1}{3}t^3$. It has the same Poincaré series as that of operads considered in \cite[Prop.~3.13]{MR4576938}, and our argument will be very similar to the one of that statement. In the polarized presentation, the generators of the ideal of relations of our operad are
\begin{gather*}
[[a_1,a_2],a_3]+[[a_2,a_3],a_1]+[[a_3,a_1],a_2],\quad
[a_1,a_2]\cdot a_3+[a_2,a_3]\cdot a_1+[a_3,a_1]\cdot a_2,\\
[a_1\cdot a_2,a_3]+[a_2\cdot a_3,a_1]+[a_3\cdot a_1,a_2],\quad
(a_1\cdot a_2)\cdot a_3+(a_2\cdot a_3)\cdot a_1+(a_3\cdot a_1)\cdot a_2,\\
2[a_1,a_2]\cdot a_3-[a_1\cdot a_3,a_2]-[[a_1,a_2],a_3]-[a_1,a_2\cdot a_3], \quad
(a_1\cdot a_2)\cdot a_3-a_1\cdot (a_2\cdot a_3)-[a_1,a_3]\cdot a_2,\\
(\gamma+\delta)a_1\cdot [a_2,a_3]+(\delta-\gamma)[a_1,[a_2,a_3]]
\end{gather*}

Suppose first that $\gamma=\delta$. In this case, our polarized presentation may be simplified to
\begin{gather*}
[[a_1,a_2],a_3]+[[a_2,a_3],a_1]+[[a_3,a_1],a_2],\quad
[a_1\cdot a_2,a_3]+[a_2\cdot a_3,a_1]+[a_3\cdot a_1,a_2],\\
-[a_1\cdot a_3,a_2]-[[a_1,a_2],a_3]-[a_1,a_2\cdot a_3], \quad
(a_1\cdot a_2)\cdot a_3,\quad
a_1\cdot [a_2,a_3]
\end{gather*}
If we consider the ordering of monomials that first compares the number of generators  $[-,-]$ used, and if these numbers coincide, compares monomials using the reverse path lexicographic order for the ordering of generators such that $[-,-]$ is greater than $-\cdot-$, one sees that our operad has a quadratic Gr\"obner basis (for instance, it follows from the fact that $[a_1\cdot a_2, a_3]$ and $[a_1\cdot a_3, a_2]$ are the only two normal monomials, and hence the quadratic part of the Gr\"obner basis gives the correct dimensions in all arities) and hence is Koszul. 

Suppose now that $\gamma\ne\delta$. 

\begin{lemma}\label{lm:lowerbound}
For $\gamma\ne\delta$, the compositional inverse of $t-t^2+\frac13t^3$ is, coefficient-wise, a lower bound for the Poincaré series of the operad $\calS_{\gamma,\delta}^!$. 
\end{lemma}

\begin{proof}Recall that the operad $\calS_{\gamma,\delta}^!$ is, up to homological shifts and linear duality, the diagonal part of the bar complex of the operad $\calS_{\gamma,\delta}$, so for the purposes of estimating the dimensions of components, we may focus on the latter chain complex. Let us fix an arity $n\ge 1$, and denote $s=\frac{\gamma+\delta}{\gamma-\delta}$. The polarized presentation of our operad is 
\begin{gather*}
[[a_1,a_2],a_3]+[[a_2,a_3],a_1]+[[a_3,a_1],a_2],\quad
[a_1,a_2]\cdot a_3+[a_2,a_3]\cdot a_1+[a_3,a_1]\cdot a_2,\\
[a_1\cdot a_2,a_3]+[a_2\cdot a_3,a_1]+[a_3\cdot a_1,a_2],\quad
(a_1\cdot a_2)\cdot a_3+(a_2\cdot a_3)\cdot a_1+(a_3\cdot a_1)\cdot a_2,\\
2[a_1,a_2]\cdot a_3-[a_1\cdot a_3,a_2]-[[a_1,a_2],a_3]-[a_1,a_2\cdot a_3], \quad
(a_1\cdot a_2)\cdot a_3-a_1\cdot (a_2\cdot a_3)-[a_1,a_3]\cdot a_2,\\
[a_1,[a_2,a_3]]-s a_1\cdot [a_2,a_3].
\end{gather*}
We shall now consider these relations as defining relations of an operad over the ring $\k[s]$. By a direct inspection, we can see that the component of arity three of this operad is a free $\k[s]$-module of rank two with a basis given by the cosets of $[a_1,a_2]\cdot a_3$ and $[a_1,a_3]\cdot a_2$. It follows that the arity $n$ component of the bar complex of our operad is a chain complex of free $\k[s]$-modules of finite rank, hence the semicontinuity theorem \cite[Sec.~III.12]{MR0463157} applies, and for each integer $k\ge 0$, the $k$-th homology of this chain complex is constant for generic $s$, and may jump up for certain special values of $s$.

Let us show that for $\gamma=-\delta$, that is for $s=0$, the Poincaré series of the operad $\calS_{\gamma,\delta}^!$ is equal to the compositional inverse of $t-t^2+\frac13t^3$. In this case, our polarized presentation may be simplified to
\begin{gather*}
[a_1,a_2]\cdot a_3+[a_2,a_3]\cdot a_1+[a_3,a_1]\cdot a_2,\quad
[a_1\cdot a_2,a_3]+[a_2\cdot a_3,a_1]+[a_3\cdot a_1,a_2],\\
(a_1\cdot a_2)\cdot a_3+(a_2\cdot a_3)\cdot a_1+(a_3\cdot a_1)\cdot a_2,\quad
2[a_1,a_2]\cdot a_3-[a_1\cdot a_3,a_2]-[a_1,a_2\cdot a_3], \\
(a_1\cdot a_2)\cdot a_3-a_1\cdot (a_2\cdot a_3)-[a_1,a_3]\cdot a_2,\quad
[a_1,[a_2,a_3]].
\end{gather*}
If we consider the ordering of monomials that first compares the number of generators $-\cdot-$ used, and if these numbers coincide, compares monomials using the reverse path lexicographic order for the ordering of generators such that $[-,-]$ is greater than $-\cdot-$, one sees that our operad has a quadratic Gr\"obner basis (for instance, it follows from the fact that $[a_1,a_2]\cdot a_3$ and $[a_1,a_3]\cdot a_2$ are the only two normal monomials, and hence the quadratic part of the Gr\"obner basis gives the correct dimensions in all arities) and hence is Koszul, and hence its compositional inverse is the Poincaré series of the operad $\calS_{\gamma,\delta}$, that is $t-t^2+\frac13t^3$.

Our discussion allows us to conclude that
\begin{itemize}
\item for generic values of $(\gamma:\delta)$, the homology of first $n$ arities of the bar complex of the operad $\calS_{\gamma,\delta}$ is concentrated on the diagonal (since the off-diagonal homology groups vanish for one specialisation $s=0$, corresponding to $\gamma=-\delta$, and since homology is semicontinuous),
\item for generic values of $(\gamma:\delta)$, the first $n$ coefficients of the Poincar\'e series of the operad $\calS_{\gamma,\delta}^!$ are  equal to the first $n$ coefficients of the compositional inverse of $t-t^2+\frac13t^3$ (since the Poincar\'e series of the bar complex of an operad  is always equal to the compositional inverse of the Poincar\'e series of that operad, and since we already know that for generic values, the homology of the first $n$ arities of the bar complex of the operad $\calS_{\gamma,\delta}$ is concentrated on the diagonal),
\item for each value of $(\gamma:\delta)$, the dimension of the $n$-th component of the operad $\calS_{\gamma,\delta}^!$ is greater than or equal to the $n$-th coefficient of the compositional inverse of $t-t^2+\frac13t^3$ (since homology is semicontinuous),
\end{itemize}
so the compositional inverse of $t-t^2+\frac13t^3$ is a lower bound for the Poincaré series of the operad $\calS_{\gamma,\delta}^!$. 
\end{proof}

\begin{lemma}\label{lm:upperbound}
For $\gamma\ne\delta$, the compositional inverse of $t-t^2+\frac13t^3$ is, coefficient-wise, an upper bound for the Poincaré series of the operad $\calS_{\gamma,\delta}^!$. 
\end{lemma}

\begin{proof}
We shall show that the shuffle tree monomials whose underlying planar trees are binary, whose vertices are labelled by the polarized operations $(-\cdot-)$ and $[-,-]$, and whose quadratic divisors do not include $[a_1,a_2]\cdot a_3$ and $[a_1,a_3]\cdot a_2$ span the Koszul dual operad. Unfortunately, one can show that there exists no quadratic Gr\"obner basis with these monomials as normal forms, so we shall use some sort of rewriting that terminates but does not have a direct meaning in terms of operads. Overall, we shall argue by induction on arity. The basis of induction is clear: since the component $\calS_{\gamma,\delta}(3)$ has a basis consisting of these monomials, the Koszul dual operad has relations allowing to express these monomials as linear combinations of others. For fixed arity, we shall take into account the label of the root vertex. Let $T$ be any shuffle tree monomial of arity $n$. If the root of $T$ is labelled by $[-,-]$, then, once we use the induction hypothesis and represent the two trees grafted at the root of $T$ as linear combinations of requested shuffle tree monomials, this immediately gives a representation of $T$ with the requisite property. Suppose that the root of $T$ is labelled by $(-\cdot-)$, so that it may have a left quadratic divisor at the root that is prohibited. Since $\alpha\ne\beta$, the two prohibited quadratic divisors appear (individually) in the two defining relations of our operad, and we may replace the arising divisor by a linear combination of allowed quadratic monomials. In the result, we may forget the monomials where the root is labelled by $[-,-]$, since we already proved our statement for such monomials. What are the other monomials that may appear? Among them there are monomials which have fewer occurrences of $[-,-]$, which we may make as another induction parameter, and monomials which have the same number of occurrences of $[-,-]$, but the arity of the left subtree of the root is smaller, which we may make another induction parameter. This means that it is possible to write $T$ as a linear combination of requested shuffle tree monomials. We already know that these monomials form a basis in the Koszul dual operad for $\alpha=\beta$, so the necessary upper bound is established.
\end{proof} 

Combining the two bounds that we found, we conclude that the Poincar\'e series of the operad $\calS_{\gamma,\delta}^!$ is the compositional inverse of $t-t^2+\frac13t^3$, so according to Proposition \ref{prop:dim2}, our operad is Koszul.
\end{proof}

\begin{remark}
The similarity of this result to \cite[Prop.~3.13]{MR4576938} that we mentioned raises a natural question: suppose that an operad $\calO$ is quotient of the magmatic operad by quadratic relations for which we have an isomorphism $\calO(3)\cong V_{2,1}$ of $S_3$-modules. Is it true that $\calO$ is Koszul? Since $V_{2,1}$ appears in the arity three component of the magmatic operad with multiplicity $4$, such quotients are parametrized by points of $\mathbb{P}^3$. The corresponding family is no longer flat: the quotient by relations (in the polarized form)
\begin{gather*}
a_1\cdot (a_2\cdot a_3)=(a_1\cdot a_2)\cdot a_3,\\
a_1\cdot[a_2,a_3]=[a_1,a_2\cdot a_3]=[a_1,[a_2,a_3]]=0
\end{gather*}
is the connected sum of the anticommutative nilpotent operad and the operad $\Com$, so it has a different Poincaré series. This breaks our proof of Koszulness, and indeed one can see that the corresponding family contains non-Koszul operads: the quotient by relations (in the polarized form)
\begin{gather*}
[a_1,[a_2,a_3]]+[a_2,[a_3,a_1]],\\
a_1\cdot[a_2,a_3]=[a_1,a_2\cdot a_3]=a_1\cdot (a_2\cdot a_3)=0
\end{gather*}
is the connected sum of the commutative nilpotent operad and the operad encoding anti-commutative associative algebras that is not Koszul \cite{MR1358617}.
\end{remark}

\begin{proposition}[Multiplicities $(0,2)$]\label{prop:seventh}
The quotient of the Novikov operad by the ideal generated by 
\begin{gather*}
(a_1,a_2,a_3)+(a_2,a_1,a_3)-2(a_3,a_1,a_2),\\
2a_1(a_2a_3)-a_3(a_1a_2)-a_3(a_2a_1)
\end{gather*}
is not Koszul.
\end{proposition}

\begin{proof}
Using the relations obtained in the proof of Theorem \ref{th:rel2}, it is immediate that all components of our operad starting from arity $4$ are one-dimensional. The first 1000 coefficients of its compositional inverse have ``good'' signs, making one suspect that this operad may be Koszul. We shall, however, show that it is not Koszul, by an argument analogous to that of \cite[Prop.~3.6]{MR4576938}. Namely, we consider the polarized presentation of our operad, for which the relations are 
\begin{gather*}
[a_1,a_2]\cdot a_3=[[a_1, a_2], a_3]=0,\\
(a_1\cdot a_2)\cdot a_3=a_1\cdot (a_2\cdot a_3),\\
[a_1\cdot a_2, a_3]+[a_1,a_2\cdot a_3]=0.
\end{gather*}  
For the weight grading $w(-\cdot-)=0$, $w([-,-])=1$, the relations are homogeneous, and so our operad inherits a weight grading. Clearly, the weighted Poincar\'e series of this operad is  
 \[
t+\frac{(1+u)}2t^2+\frac{1+u}{6}t^3+\sum_{k\ge 4}\frac{t^k}{k!}.
 \]
For the compositional inverse of this power series, the coefficient at~$t^{20}$ has a positive coefficient $\frac{14119421138089}{17322439680000}$ at $u^2$, so our operad is not Koszul.
\end{proof}

\begin{proposition}[Multiplicities $(1,2)$]\label{prop:eighth}
The quotient of the Novikov operad by the ideal generated by 
\begin{gather*}
\alpha((a_1a_2)a_3+(a_2a_1)a_3+(a_3a_1)a_2)+\beta(a_1(a_2a_3)+a_2(a_3a_1)+a_3(a_1a_2)),\\
(a_1,a_2,a_3)+(a_2,a_1,a_3)-2(a_3,a_1,a_2),\\
2a_1(a_2a_3)-a_3(a_1a_2)-a_3(a_2a_1)
\end{gather*}
is Koszul. 
\end{proposition}

\begin{proof}
From Proposition \ref{prop:seventh}, we know that the last two relations define an operad whose polarized presentation is  
\begin{gather*}
[a_1,a_2]\cdot a_3=[[a_1, a_2], a_3]=0,\\
(a_1\cdot a_2)\cdot a_3=a_1\cdot (a_2\cdot a_3),\\
[a_1\cdot a_2, a_3]+[a_1,a_2\cdot a_3]=0.
\end{gather*}  
Since 
 \[
\alpha((a_1a_2)a_3+(a_2a_3)a_1+(a_3a_1)a_2)+\beta(a_1(a_2a_3)+a_2(a_3a_1)+a_3(a_1a_2))
 \]
is a linearization of $\alpha a^2a+\beta aa^2=\frac12(\alpha-\beta) a^2\cdot a + \frac12(\alpha+\beta)[a^2, a]$, it is equivalent, modulo the first two relations, to 
 \[
(\alpha+\beta) a_1\cdot (a_2\cdot a_3) + (\beta-\alpha) [a_1,a_2\cdot a_3]
 \]
If $\alpha+\beta=0$, the relations mean that the operation $(-\cdot-)$ is associative, the operation $[-,-]$ is two-step nilpotent, and all compositions of these operations with one another vanish. This means that we are dealing with the connected sum of the operad of commutative associative algebras and the operad of anticommutative two-step nilpotent algebras. These two operads are well known to be Koszul, and so their connected sum is Koszul too (it follows from the fact that the bar complex of the connected sum is the coproduct of bar complexes).
If $\alpha+\beta\ne 0$, our new relation is of the form $a_1\cdot (a_2\cdot a_3) = t [a_1,a_2\cdot a_3]$ for some $t$. If we consider the ordering of monomials that first compares the number of generators $-\cdot-$ used, and if these numbers coincide, compares monomials using the path lexicographic order (for either order of generators), one sees that our operad has a quadratic Gr\"obner basis and hence is Koszul. 
\end{proof}

\begin{proposition}[Multiplicities $(2,2)$]\label{prop:ninth}
The quotient of the Novikov operad by its arity three component is Koszul. 
\end{proposition}

\begin{proof}
This is the operad of nilpotent algebras of index three which is well known to be Koszul. 
\end{proof}

The results of Propositions \ref{prop:first}--\ref{prop:ninth} imply the following theorem. 

\begin{theorem}\label{th:Koszul}
The following Koszul operads with one binary generator admit the (right) Novikov operad as a quotient:
\begin{itemize}
\item the operad of (left) nonassociative permutative algebras $\NAP$ defined by the identity $a_1(a_2a_3)-a_2(a_1a_3)=0$,
\item the (right) pre-Lie operad defined by the identity $(a_1,a_2,a_3)=(a_1,a_3,a_2)$,
\item each operad in the parametric family depending on the parameter $(\gamma:\delta)\in\mathbb{P}^1$ defined by the identity 
\begin{multline*}
\gamma((a_1,a_2,a_3)+(a_3,a_2,a_1)-(a_2,a_1,a_3)-(a_2,a_3,a_1))+\\
\delta((a_1a_2)a_3+(a_3a_2)a_1-(a_1a_3)a_2-(a_3a_1)a_2),
\end{multline*}
\item each operad in the parametric family depending on the parameter $(\alpha:\beta)\in\mathbb{P}^1$ defined by the identity 
\begin{multline*}
\alpha((a_1a_2)a_3+(a_2a_3)a_1+(a_3a_1)a_2+(a_1a_3)a_2+(a_2a_1)a_3+(a_3a_2)a_1)+\\
\beta(a_1(a_2a_3)+a_2(a_3a_1)+a_3(a_1a_2)+a_1(a_3a_2)+a_2(a_1a_3)+a_3(a_2a_1)).
\end{multline*}
\item the magmatic operad of absolutely free nonassociative algebras.
\end{itemize}
\end{theorem}

\begin{proof}
First, we note that from Propositions \ref{prop:first}--\ref{prop:ninth} we obtain a complete list of Koszul quotients of the left Novikov operad. Computing Koszul duals, we obtain a complete list of Koszul operads admitting the right Novikov operad as a quotient. The corresponding computations are straightforward and omitted.
\end{proof}

\begin{remark}
Recall that a Lie-admissible algebra \cite{MR27750} is an algebra with one binary operation for which the commutator $[a,b]=ab-ba$ satisfies the Jacobi identity. In the $(\alpha:\beta)$ parametric family of Theorem \ref{th:Koszul}, the operad for $(\alpha:\beta)=(1:-1)$ is easily seen to be the operad of Lie-admissible algebras. The defining identity of that operad is a skew-symmetric identity of arity three; other identities of that type were studied by Dzhumadildaev \cite{MR2676255} who in particular introduced alia (anti-Lie admissible) algebras that are defined, in the polarized form, by the identity
 \[
[a_1,a_2]\cdot a_3+[a_2,a_3]\cdot a_1+[a_3,a_1]\cdot a_2=0,
 \]
 and left alia algebras that are defined, in the polarized form, by the identity
 \[
[a_1,a_2] a_3+[a_2,a_3] a_1+[a_3,a_1] a_2=0.
 \]
One can also check that the operad for $(\alpha:\beta)=(1:1)$ of that family is the operad of alia algebras, and the operad for any other value of $(\alpha:\beta)$ is isomorphic to the operad of left alia algebras,
\end{remark}

\appendix

\section{Quotient by a copy of the two-dimensional module: technical lemmas}

\subsection{Proof of Lemma \ref{lm:symmetryI}}\label{app:symmetryI}

The proof consists of several lemmas that progress by gradual multilinearization. Recall that throughout this section we may use Equation \eqref{eq:I}.

\begin{lemma}\label{lm:sym1aaab}
We have 
\begin{equation}\label{eq:I13}
a''b'a^2-a''a'ab=0.
\end{equation}
\end{lemma}

\begin{proof}
Multiplying \eqref{eq:I} by $a'$, we obtain 
\begin{equation}\label{eq:I06}
a''a'ab-b''a'a^2+\delta ((a')^3b-(a')^2b'a)=0.
\end{equation}
Applying the derivation $\Delta_{a\mapsto a'a}$ to \eqref{eq:I}, we obtain 
\begin{equation}\label{eq:I08}
a'''a^2b+4a''a'ab-2b''a'a^2+\delta (2a''a'ab+2(a')^3b-a''b'a^2-2(a')^2b'a)=0.
\end{equation}
Subtracting twice \eqref{eq:I06} from \eqref{eq:I08}, we obtain
\begin{equation}\label{eq:I10}
a'''a^2b+2a''a'ab+\delta (2a''a'ab-a''b'a^2)=0.
\end{equation}
Substituting $b:=b'a$ into \eqref{eq:I}, we obtain
\begin{equation}\label{eq:I07}
b'''a^3+2 b''a'a^2+\delta b''a'a^2=0.  
\end{equation}
Taking the derivative of \eqref{eq:I} and multiplying by $a$, we obtain
\begin{equation}\label{eq:I09}
a'''a^2b+a''b'a^2+a''a'ab-b'''a^3-2b''a'a^2+\delta (2a''a'ab-a''b'a^2-b''a'a^2)=0.
\end{equation}
Adding \eqref{eq:I07} to \eqref{eq:I09}, we obtain
\begin{equation}\label{eq:I11}
a'''a^2b+(1+2\delta)a''a'ab+(1-\delta)a''b'a^2=0.
\end{equation}
Subtracting \eqref{eq:I11} from \eqref{eq:I10}, we obtain
 \[
a''a'ab-a''b'a^2=0.
 \]
\end{proof}

\begin{lemma}
We have 
\begin{equation}\label{eq:I28}
a''b'ab-a''a'b^2=0.
\end{equation}
\end{lemma}

\begin{proof}
Substituting $b:=b'b$ into \eqref{eq:I}, we obtain 
\begin{equation}\label{eq:I17}
b'''ba^2+3b''b'a^2-a''b'ab+\delta (b''a'ab+a'(b')^2a-(a')^2b'b)=0.
\end{equation}
Applying the derivation $\Delta_{a\mapsto a'b}$ to \eqref{eq:I}, we obtain 
\begin{equation}\label{eq:I18}
a'''ab^2+2a''b'ab-b''a'ab+a''a'b^2+\delta (2a''a'b^2-a''b'ab+(a')^2b'b-a'(b')^2a)=0.
\end{equation}
Multiplying \eqref{eq:I} by $b'$, we obtain
\begin{equation}\label{eq:I19}
a''b'ab-b''b'a^2+\delta ((a')^2b'b-a'(b')^2a)=0.
\end{equation}
Subtracting \eqref{eq:I19} from \eqref{eq:I18}, one obtains 
\begin{equation}\label{eq:I20}
a'''ab^2+(1-\delta)a''b'ab-b''a'ab+(1+2\delta)a''a'b^2+b''b'a^2=0.
\end{equation}
Adding \eqref{eq:I19} to \eqref{eq:I17}, we obtain 
\begin{equation}\label{eq:I21}
b'''a^2b+2b''b'a^2+\delta b''a'ab=0.
\end{equation}
Interchanging $a$ and $b$ in \eqref{eq:I21}, we obtain
\begin{equation}\label{eq:I22}
a'''ab^2+2a''a'b^2+\delta a''b'ab=0.
\end{equation}
The difference of \eqref{eq:I20} and \eqref{eq:I22} is
\begin{equation}\label{eq:I23}
(-1 +2 \delta) a''a'b^2+(1-2\delta) a''b'ab-b''a'ab+b''b'a^2=0. 
\end{equation}
Taking the derivative of \eqref{eq:I} and multiplying by $b$, we obtain
\begin{equation}\label{eq:I24}
a'''ab^2-b'''a^2b+(1-\delta)a''b'ab+(1+2\delta)a''a'b^2-(2+\delta)b''a'ab=0.
\end{equation}
A linear combination of \eqref{eq:I21}, \eqref{eq:I22}, and \eqref{eq:I24} gives us
\begin{equation}\label{eq:I25}
(-1+2\delta) a''a'b2+(1 - 2\delta) a''b'ab-2 b''a'ab+2b''b'a^2=0.
\end{equation}
Finally, the difference of \eqref{eq:I23} and \eqref{eq:I25} is
\begin{equation}\label{eq:I26}
b''a'ab- b''b'a^2=0,
\end{equation}
which is what we want up to renaming variables.
\end{proof}

\begin{lemma}
We have
\begin{equation}\label{eq:I39}
c''a'ab-c''b'a^2=0.
\end{equation}
\end{lemma}

\begin{proof}
Multiplying \eqref{eq:I} by $c'$, we obtain
\begin{equation}\label{eq:I29}
a''c'ab-b''c'a^2+\delta ((a')^2c'b-a'b'c'a)=0.
\end{equation}
Applying the derivation $\Delta_{a\mapsto a'c}$ to \eqref{eq:I}, we obtain 
\begin{multline}\label{eq:I31}
a'''abc+2a''c'ab+a''a'bc-2b''a'ac+c''a'ab+\\ \delta (2a''a'bc+2(a')^2c'b-a''b'ac-a'b'c'a-(a')^2b'c)=0.
\end{multline}
Substituting $b:=b'c$ into \eqref{eq:I}, we obtain 
\begin{equation}\label{eq:I32}
b'''a^2c-a''b'ac+2b''c'a^2+c''b'a^2+\delta (b''a'ac-(a')^2b'c+a'b'c'a)=0.
\end{equation}
Taking the derivative of \eqref{eq:I} and multiplying by $c$, we obtain
\begin{equation}\label{eq:I33}
a'''abc-b'''a^2c+a''b'ac+a''a'bc-2b''a'ac+\delta (2a''a'bc-a''b'ac-b''a'ac)=0.
\end{equation}
Taking linear combinations of equations \eqref{eq:I29}--\eqref{eq:I33}, we obtain \eqref{eq:I39}.
\end{proof}

We are now ready to prove the identity of Lemma \ref{lm:symmetryI}. Applying the derivation $\Delta_{a\mapsto d}$ to \eqref{eq:I39}, we obtain 
\begin{equation}\label{eq:I56}
c''d'ab+c''a'bd-2c''b'ad=0.
\end{equation}
Interchanging $c$ and $d$ in \eqref{eq:I56} gives us
\begin{equation}\label{eq:I57}
d''c'ab+d''a'bc-2d''b'ac=0. 
\end{equation}
Applying the derivations $\Delta_{a\mapsto c}$ and $\Delta_{b\mapsto d}$ to \eqref{eq:I26}, we obtain 
\begin{equation}\label{eq:I58}
d''c'ab+b''c'ad+d''a'bc+b''a'cd-2d''b'ac-2b''d'ac=0.
\end{equation}
From \eqref{eq:I57} and \eqref{eq:I58}, we obtain
\begin{equation}\label{eq:I59}
b''c'ad+b''a'cd-2b''d'ac=0. 
\end{equation}
Interchanging $b$ and $c$ in \eqref{eq:I56} gives us 
\begin{equation}\label{eq:I60}
b''d'ac+b''a'cd-2b''c'ad=0.
\end{equation}
Finally, from \eqref{eq:I59} and \eqref{eq:I60} we obtain $b''c'ad-b''d'ac=0$, 
which is what we want up to renaming variables.

\subsection{Proof of Lemma \ref{lm:symmetryII}}\label{app:symmetryII}

The proof consists of several lemmas that progress by gradual multilinearization. Recall that throughout this section we may use Equation \eqref{eq:II}.

Substituting $b:=a'a$, $b:=b'a$, and $b:=a'b$ into \eqref{eq:II}, we obtain 
\begin{equation}\label{eq:I73}
a''a'a^2=b''a'a^2=a''a'ab=0.
\end{equation}
Taking the derivative of \eqref{eq:II} and multiplying by $a$, we obtain 
 \[
2a''a'ab-a''b'a^2-b''a'a^2=0,
 \]
which, together with what we already proved, implies $a''b'a^2=0$.

Substituting $b:=b'b$ into \eqref{eq:II} gives us
 \[
b''a'ab+a'(b')^2a-(a')^2b'b=0,
 \]
which, modulo Equation \eqref{eq:II}, becomes simply $b''a'ab=0$. Applying the derivation $\Delta_{a\mapsto a'b}$ to \eqref{eq:II}, we obtain 
 \[
2a''a'b^2-a''b'ab+(a')^2b'b-a'(b')^2a=0,
 \]
which, modulo Equation \eqref{eq:II}, becomes simply $2a''a'b^2-a''b'ab=0$, which, using the relation $b''a'ab=0$ that we already proved (with $a$ and $b$ interchanged), implies $a''a'b^2=0$.

Multiplying \eqref{eq:II} by $c'$, we obtain
 \[
(a')^2c'b-a'b'c'a=0.
 \]
Interchanging $b$ and $c$ in that latter equation, we obtain
 \[
(a')^2b'c-a'b'c'a=0.
 \]
Substituting $b:=b'c$ into \eqref{eq:II}, we obtain 
 \[
b''a'ac-(a')^2b'c+a'b'c'a=0.
 \]
The two latter equations imply $b''a'ac=0$. Identity $a''a'bc=0$ is symmetric in $b,c$ and hence follows from already proved $a''a'b^2=0$.
Taking the derivative of \eqref{eq:II} and multiplying by $c$, we obtain
 \[
2a''a'bc-a''b'ac-b''a'ac=0,
 \]
which, using the already proved identities, implies $a''b'ac=0$. Applying the derivation $\Delta_{a\to c}$ to $b''a'a^2=0$ and using the already proved identities gives us $b''c'a^2=0$. Finally, since the identity $a''b'cd=0$ is symmetric in $c,d$, it follows from its versions where at most three letters are different. 

\section{A family of two-dimensional Novikov algebras}\label{app:2D}

Let us consider the algebra $B_\delta$ with a basis $e,f$ and the multiplication table
 \[
ee=0, \quad ef=-\delta e,\quad
fe=e, \quad ff=f.
 \]

\begin{proposition}\label{prop:Bdelta}
The algebra $B_\delta$ is a Novikov algebra. Moreover, the identity 
 \[
((a,a,b)-(b,a,a))+\delta(a(ab)-ba^2)=0
 \]
holds in this algebra.
\end{proposition}

\begin{proof}
Let us first check that this is a Novikov algebra. We compute the triple products 
\begin{gather*}
(ee)e=0,\quad e(ee)=0,\quad
(ef)e=0,\quad e(fe)=0,\\
(fe)e=0,\quad f(ee)=0,\quad
(ff)e=e,\quad f(fe)=e,\\
(ee)f=0,\quad e(ef)=0,\quad (ef)f=\delta^2 e,\quad e(ff)=-\delta e,\\
(fe)f=-\delta e,\quad f(ef)=-\delta e,\quad
(ff)f=f,\quad f(ff)=f,
\end{gather*} 
and see that the only associator that is nonzero if $(e,f,f)$ which is automatically symmetric, and that the only nontrivial left-commutativity to check is $e(ff)=f(ef)$ which does indeed hold. 

The multilinear form of $((a,a,b)-(b,a,a))+\delta(a(ab)-ba^2))=0$ is
 \[
((a,c,b)+(c,a,b)-(b,a,c)-(b,c,a))+\delta(a(cb)+c(ab)-b(ac)-b(ca))=0.
 \]
This identity is symmetric in $a,c$, so there are just the following choices for $(a,b,c)$ to check:
\begin{itemize}
\item case $(e,e,e)$: the identity trivially holds (all terms are zero), 
\item case $(e,e,f)$: the identity trivially holds (all terms are zero),
\item case $(f,e,f)$: the identity becomes  
 \[
((f,f,e)+(f,f,e)-(e,f,f)-(e,f,f))+\delta(f(fe)+f(fe)-e(ff)-e(ff))=0,
 \] that is $-2(\delta^2+\delta)e+\delta(2e+2\delta e)$, which is true,
\item case $(e,f,e)$: the identity trivially holds (all terms are zero), 
\item case $(e,f,f)$: the identity becomes 
  \[
((e,f,f)+(f,e,f)-(f,e,f)-(f,f,e))+\delta(e(ff)+f(ef)-f(ef)-f(fe)),
 \] 
that is $(\delta^2+\delta)e+\delta(-\delta e-e)$, which is true, 
\item case $(f,f,f)$: the identity trivially holds. 
\end{itemize}
\end{proof}

\begin{remark}
Novikov algebras of dimension two were classified by Bai and Meng in \cite{MR1818753}; the algebra we consider is, up to a sign of the parameter, the algebra N6 in their classification.  
\end{remark}

\section{Three operadic Gr\"obner bases calculations}\label{app:Grobner}

In this section, we shall show that, for a good choice of generators, some quotients of the Novikov operad have a finite Gr\"obner basis of relations. We begin with recording the following lemma that can be checked by a direct calculation.

\begin{lemma}
Consider the operations $a\cdot b=ab+ba$ and $[a,b]=ab-ba$ in a Novikov algebra. These operations satisfy the identities
\begin{gather*}
[[a_1,a_2],a_3]=[a_1,[a_2,a_3]]+[[a_1,a_3],a_2],\\
[a_1\cdot a_2,a_3]=[a_1\cdot a_3,a_2]+2a_1\cdot[a_2,a_3]+[a_1,[a_2,a_3]],\\
2[a_1,a_2]\cdot a_3=[a_1\cdot a_3,a_2]+[[a_1,a_3],a_2]+[a_1,a_2\cdot a_3]+[a_1,[a_2,a_3]], \\
2(a_1\cdot a_3)\cdot a_2=[a_1\cdot a_3,a_2]+[[a_1,a_3],a_2]+2a_1\cdot (a_2\cdot a_3)+[a_1,a_2\cdot a_3]+[a_1,[a_2,a_3]],\\
2[a_1,a_3]\cdot a_2=[a_1\cdot a_3,a_2]+[[a_1,a_3],a_2]+2a_1\cdot[a_2,a_3]+[a_1,a_2\cdot a_3]+[a_1,[a_2,a_3]],\\
2(a_1\cdot a_2)\cdot a_3=[a_1\cdot a_3,a_2]+[[a_1,a_3],a_2]+2a_1\cdot(a_2\cdot a_3)+2a_1\cdot[a_2,a_3]+[a_1,a_2\cdot a_3]+[a_1,[a_2,a_3]].
\end{gather*}
\end{lemma}

The following three lemmas are proved using the operadic Gr\"obner basis calculator \cite{OpGb}. In each of them, we use the operations $a\cdot b=ab+ba$ and $[a,b]=ab-ba$, and consider the same reverse graded reverse path lexicographic order of monomials, assuming that $[-,-]$ is greater than $-\cdot-$.

\begin{lemma}
The operad $\calQ_{1,-1}$ has the following reduced Gr\"obner basis of relations:
\begin{gather*}
[a_1,[a_2,a_3]],\quad
[[a_1,a_3],a_2],\quad
[[a_1,a_2],a_3],\\
a_1\cdot[a_2,a_3]+1/2[a_1\cdot a_3,a_2]-1/2[a_1\cdot a_2,a_3],\quad
[a_1,a_2\cdot a_3]-2[a_1,a_2]\cdot a_3+[a_1\cdot a_3,a_2],\\
a_1\cdot(a_2\cdot a_3)-(a_1\cdot a_2)\cdot a_3+[a_1,a_2\cdot a_3]-1/2[a_1\cdot a_3,a_2]+1/2[a_1\cdot a_2,a_3],\\
(a_1\cdot a_3)\cdot a_2-(a_1\cdot a_2)\cdot a_3-1/2[a_1\cdot a_3,a_2]+1/2[a_1\cdot a_2,a_3],\\
[a_1,a_3]\cdot a_2-[a_1,a_2]\cdot a_3+1/2[a_1\cdot a_3,a_2]-1/2[a_1\cdot a_2,a_3],\\
[[a_1,a_2]\cdot a_4,a_3],\quad
[[a_1,a_2]\cdot a_3,a_4],\quad
[[a_1,a_3]\cdot a_4,a_2],\\
[a_1\cdot a_3,a_2]\cdot a_4-[a_1\cdot a_2,a_3]\cdot a_4-[(a_1\cdot a_3)\cdot a_4,a_2]+[(a_1\cdot a_2)\cdot a_4,a_3],\\
 [a_1\cdot a_4,a_2]\cdot a_3-[a_1\cdot a_2,a_3]\cdot a_4-[(a_1\cdot a_3)\cdot a_4,a_2]+\frac12[(a_1\cdot a_2)\cdot a_4,a_3]+\frac12[(a_1\cdot a_2)\cdot a_3,a_4],\\
([a_1,a_2]\cdot a_3)\cdot a_4-2[a_1\cdot a_2,a_3]\cdot a_4-[(a_1\cdot a_3)\cdot a_4,a_2]+\frac32[(a_1\cdot a_2)\cdot a_4,a_3]+\frac12[(a_1\cdot a_2)\cdot a_3,a_4].
\end{gather*} 
\end{lemma}

\begin{lemma}
The operad $\calQ_{1,0}$ has the following reduced Gr\"obner basis of relations:
\begin{gather*}
a_1\cdot[a_2,a_3]+[[a_1,a_3],a_2]-[[a_1,a_2],a_3],\quad
[a_1,a_2\cdot b_3]+[a_1\cdot a_2,a_3]+3[[a_1,a_3],a_2],\\
a_1\cdot(a_2\cdot a_3)-(a_1\cdot a_2)\cdot a_3-[[a_1,a_3],a_2],\quad
[a_1,[a_2,a_3]]+[[a_1,a_3],a_2]-[[a_1,a_2],a_3,]\\
[a_1,a_2]\cdot a_3+[[a_1,a_2],a_3],\quad
[a_1,a_3]\cdot a_2+[[a_1,a_3],a_2],\\
(a_1\cdot a_3)\cdot a_2-(a_1\cdot a_2)\cdot a_3+[[a_1,a_3],a_2]-[[a_1,a_2],a_3],\\
[a_1\cdot a_3,a_2]-[a_1\cdot a_2,a_3]+3[[a_1,a_3],a_2]-3[[a_1,a_2],a_3],\\
[[a_1\cdot a_2,a_4],a_3]-[[a_1\cdot a_2,a_3],a_4] -2[[[a_1,a_4],a_2],a_3]+2[[[a_1,a_3],a_2],a_4], \\
[[a_1\cdot a_3,a_4],a_2]-[[a_1\cdot a_2,a_3],a_4]+2[[[a_1,a_4],a_2],a_3]-3[[[a_1,a_3],a_2],a_4]-[[[a_1,a_2],a_3],a_4],\\
[(a_1\cdot a_2)\cdot a_3,a_4]+3[[a_1\cdot a_2,a_3],a_4]-4[[[a_1,a_4],a_2],a_3]+8[[[a_1,a_3],a_2],a_4]-2[[[a_1,a_2],a_3],a_4],\\  
[[[a_1,a_4],a_3],a_2]-[[[a_1,a_4],a_2],a_3],\quad
[[[a_1,a_3],a_4],a_2]-[[[a_1,a_3],a_2],a_4],\\
[[[a_1,a_2],a_4],a_3]-[[[a_1,a_2],a_3],a_4].
\end{gather*} 
\end{lemma}

\begin{lemma}
The operad $\calQ_{0,1}$ has the following reduced Gr\"obner basis of relations:
\begin{gather*}
a_1\cdot [a_2,a_3]-[[a_1,a_3],a_2]+[[a_1,a_2],a_3],\quad
[a_1,a_2\cdot a_3]-[[a_1,a_3],a_2]+[a_1\cdot a_2,a_3],\\
a_1\cdot (a_2\cdot a_3)+[[a_1,a_3],a_2]-(a_1\cdot a_2)\cdot a_3,\quad
[a_1,[a_2,a_3]]+[[a_1,a_3],a_2]-[[a_1,a_2],a_3],\\
[a_1,a_2]\cdot a_3 - [[a_1,a_2],a_3],\quad
[a_1,a_3]\cdot a_2 - [[a_1,a_3],a_2],\\
(a_1\cdot a_3)\cdot a_2-(a_1\cdot a_2)\cdot a_3+[[a_1,a_3],a_2]-[[a_1,a_2],a_3],\\
[a_1\cdot a_3,a_2]-[a_1\cdot a_2,a_3]+[[a_1,a_3],a_2]-[[a_1,a_2],a_3],\\
[[a_1\cdot a_3,a_4],a_2]-[[a_1\cdot a_2,a_3],a_4]+[[[a_1,a_3],a_2],a_4]-[[[a_1,a_2],a_3],a_4],\\
[[a_1\cdot a_2,a_4],a_3]-[[a_1\cdot a_2,a_3],a_4],\quad
[(a_1\cdot a_2)\cdot a_3,a_4]-[[a_1\cdot a_2,a_3],a_4], \\
[[[a_1,a_4],a_3],a_2]-[[[a_1,a_4],a_2],a_3],\quad
[[[a_1,a_3],a_4],a_2]-[[[a_1,a_3],a_2],a_4],\\
[[[a_1,a_2],a_4],a_3]-[[[a_1,a_2],a_3],a_4].
\end{gather*} 
\end{lemma}

\begin{corollary}\label{cor:dim}
The dimension of the components $\calQ_{1,-1}(n)$, $\calQ_{1,0}(n)$, and $\calQ_{0,1}(n)$ is $n+1$. 
\end{corollary}

\begin{proof}
Let us start with $\calQ_{1,-1}(n)$. The first four elements of the Gr\"obner basis ensure that all normal forms are left combs. Examining the other relations, we see that the normal monomials of arity at least $4$ are precisely the following ones:
\begin{gather*}
(\cdots (a_1\cdot a_2)\cdots a_{n-1})\cdot a_n,\\
[(\cdots((\cdots (a_1\cdot a_2)\cdots a_{i-1})\cdot a_i)\cdots a_n),a_i], \quad 2\le i\le n,\\
[(\cdots(\cdots (a_1\cdot a_2)\cdots )\cdot a_{n-2}),a_{n-1}]\cdot a_n,
\end{gather*}
and there are exactly $n+1$ of them.

The leading terms of the Gr\"obner bases of $\calQ_{1,0}$ and $\calQ_{0,1}$ are the same, so we shall prove these two statements simultaneously. 
There are four relations that ensure that all normal forms are left combs. Examining the other relations, we see that the normal monomials of arity at least $4$ are precisely the following ones:

\begin{gather*}
(\cdots (a_1\cdot a_2)\cdots a_{n-1})\cdot a_n,\\
[\cdots[[\cdots [[[a_1,a_i],a_2],a_3],\cdots,a_{i-1}],a_{i+1}],\cdots,a_n], \quad 2\le i\le n,\\
[[\cdots [(a_1\cdot a_2),a_3],\cdots a_{n-1}], a_n],
\end{gather*}
and there are exactly $n+1$ of them.
\end{proof}

\printbibliography

\end{document}